\documentclass{amsart}

\usepackage[latin1]{inputenc}
\usepackage[english]{babel}
\usepackage{amstext}
\usepackage{latexsym}
\usepackage{amsmath}
\usepackage{amssymb}
\usepackage{amsfonts}
\usepackage{graphicx}
\usepackage{amssymb}
\usepackage{amsthm}
\usepackage{latexsym}
\usepackage{epsfig}
\usepackage{tikz}
\usetikzlibrary{matrix,arrows,decorations.pathmorphing}
\usepackage{graphicx}
\usepackage{enumitem}
\DeclareRobustCommand{\SkipTocEntry}[4]{}

\input xy
\xyoption{all}
\begin{document}

\newtheorem{theorem}{Theorem}[section]
\newtheorem*{teoA}{Theorem A}
\newtheorem*{teoB}{Theorem A'} 
\newtheorem*{teoC}{Theorem B}
\newtheorem*{teoD}{Theorem B'}
\newtheorem*{nuevo}{Theorem C}
\newtheorem*{nuevo2}{Theorem D}
\newtheorem*{invariance}{Invariance Principle}
\newtheorem{defn}[theorem]{Definition}
\newtheorem*{defn*}{Definition}
\newtheorem{remark}[theorem]{Remark}
\newtheorem{claim}{Claim}
\newtheorem{lemma}[theorem]{Lemma}
\newtheorem{prop}[theorem]{Proposition}
\newtheorem*{cor-intro}{Corollary}
\newtheorem{cor}[theorem]{Corollary}
\newtheorem{cor-app}{Corollary}
\newtheorem{obs}[theorem]{Remark}
\newtheorem*{propf}{Proposition 5.2}
\newtheorem*{propu}{Proposition 5.1}

\title[$C^1$-openness of non-uniform hyperbolicity]{$C^1$-openness of non-uniform hyperbolic diffeomorphisms with bounded $C^2$ norm}
\author{Chao Liang}
\thanks{C.L. has been supported by NNSFC(\#11471344) and CUFE Young Elite Teacher Project (\#QYP1705).}
\address{School of Statistics and Mathematics, Central University of Finance and Economics, Beijing, 100081, China}
\email{chaol@cufe.edu.cn}

\author{Karina Marin}
\thanks{K.M. has been supported by INCTMat-CAPES during her postdoctoral fellowship at PUC-Rio.}
\address{Departamento de Matem\'atica, Instituto de Ci\^encias Exatas (ICEx), Universidade Federal de Minas Gerais, Brazil}
\email{kmarin@mat.ufmg.br}

\author{Jiagang Yang}
\thanks{J.Y. has been partially supported by CNPq, FAPERJ, and PRONEX}
\address{Departamento de Geometria, Instituto de Matem\'atica e Estat\'istica, Universidade Federal Fluminense, Niter\'oi, Brazil}
\email{yangjg@impa.br}

\begin{abstract}
We study the $C^1$-topological properties of the subset of non-uniform hyperbolic diffeomorphisms in a certain class of $C^2$ partially hyperbolic symplectic systems which have bounded $C^2$ distance to the identity. In this set, we prove the stability of non-uniform hyperbolicity as a function of the diffeomorphism and the measure, and the existence of an open and dense subset of continuity points for the center Lyapunov exponents. These results are generalized to the volume-preserving context. 
\end{abstract}

	
\maketitle

\section{Introduction}

In the 1970's, Pesin proposed a more flexible notion of hyperbolicity called \textit{non-uniform hyperbolicity}. A diffeomorphism is said to be non-uniform hyperbolic if all its Lyapunov exponents are different from zero almost everywhere with respect to some preferred invariant measure, for instance, a volume measure. While being more general, non-uniform hyperbolicity still has many important consequences. Most notably: the stable manifold theorem (Pesin \cite{P}), the abundance of periodic points and Smale horseshoes (Katok \cite{K}) and the fact that the fractal dimension of invariant measures is well defined (Ledrappier and Young \cite{LY2} and Barreira, Pesin and Schmelling \cite{BPS}).

In order to understand the properties of the subset of non-uniform hyperbolic systems, we need to understand how the Lyapunov exponents vary with the diffeomorphism in the different topologies.

For the $C^1$ case, Ma\~n\'e observed that every area-preserving diffeomorphism that is not Anosov can be $C^1$-approximated by diffeomorphisms
with zero Lyapunov exponents. His arguments were completed by Bochi \cite{B1} and were extended to arbitrary dimension by Bochi and Viana \cite{B2,BV1}. In particular, Bochi \cite{B2} proved that the set of non-uniform hyperbolic diffeomorphisms is not $C^1$ open among partially hyperbolic symplectic systems. 

Recently, the authors proved in \cite{LMY} that the situation is very different when we have more regular systems. More precisely, they consider a $C^2$ open subset of partially hyperbolic symplectic diffeomorphisms with 2-dimensional center bundle, denoted by $\mathcal{B}^2_{\omega}(M)$, and proved that the subset of non-uniform hyperbolic systems in $\mathcal{B}^2_{\omega}(M)$ is $C^2$-open. 

In this paper, we are interested in the behavior between these two cases. We proved that in order to obtain $C^1$-openess of the subset of non-uniform hyperbolic diffeomorphisms, we need to ask for them to be at a $C^2$-bounded distance to the identity. In other words, the phenomenon observed by Bochi in \cite{B2} can only be achieved by an explosion of the $C^2$ norm. 

For every fixed $N>0$, we denote $\mathcal{E}^N_{\omega}(M)$ a subset of $C^2$ partially hyperbolic symplectic diffeomorfisms, where $f\in \mathcal{E}^N_{\omega}(M)$ if it is accessible, has 2-dimensional center bundle, satisfy certain pinching and bunching conditions and $dist_{C^2}(f,id)<N$. (All the key words will be defined in the next section.)

\begin{teoA} The subset of non-uniformly hyperbolic diffeomorphisms in $\mathcal{E}^N_{\omega}(M)$ is $C^1$-open.
\end{teoA}

Moreover, the same feature of continuity of the center Lyapunov exponents proved in \cite{LMY} can be extended to this context.

\begin{teoC} For every $r\geq 2$, there exists a subset $\mathcal{U}$ which is $C^1$-open and $C^r$-dense in $\mathcal{E}^N_{\omega}(M)$ and such that every $g\in \mathcal{U}$ is a $C^1$-continuity point for the center Lyapunov exponents.  
\end{teoC}

A classical result of Furstenberg provides a relation between the center Lyapunov exponents of $f$ and the invariant measures of the derivative cocycle $\mathbb{P}(F)$ where $F=Df\vert E^c$. Therefore, the study of the center Lyapunov exponents and the study of $\mathbb{P}(F)$-invariant measures is intermingled. Among all the $\mathbb{P}(F)$-invariant measures we are more interested in the ones that satisfies certain properties called $s$-state and $u$-state. (We refer the reader to next section for precise definitions).

One of the key results that allow us to conclude Theorem A and Theorem B is Corollary \ref{limitu}. Roughly speaking, it states that the limit of $u$-states is a $u$-state. This result is typical in the theory of linear cocycles and recently has been extended to the context of derivative cocycles in \cite{OP}. 

We provide a different proof of Corollary \ref{limitu} using a criterion for $u$-states introduced by Tahzibi and Yang in \cite{TY}. This criterion establishes that a measure is a $u$-state if, and only if, some inequality involving the entropies along the expanding foliations of $f$ and $\mathbb{P}(F)$ is in fact an equality. In order to conclude the result we are going to prove that the partial entropy of $\mathbb{P}(F)$ varies upper semi-continuously. This has been proved for $C^1$ diffeomorphism by Yang in \cite{Y}. We could have applied this theorem, if the cocycle $\mathbb{P}(F)$ was also $C^1$. However, the derivative cocycle is only H\"older continuous and we need to adjust the techniques to our context. 

This new approach to study $u$-states allow us to conclude the following result. Let $\mu$ denote the Lebesgue measure associated to $\omega$.  

\begin{nuevo} Let $f\in \mathcal{E}^N_{\omega}(M)$ be a non-uniform hyperbolic diffeomorphism that does not admit $su$-states. Then, for every sequence of ergodic measures $\nu_k$ such that $\nu_k\to \mu$ in the weak$^{*}$ topology and for $i=s$ or $i=u$ verifies
\begin{equation*}
\lim_{k\to \infty} h_{\nu_k}(f,W^i)=h_{\mu}(f,W^i),
\end{equation*}
there exists $K\in \mathbb{N}$ such that $\nu_k$ is a hyperbolic measures of $f$ for every $k\geq K$. 
Moreover, the center Lyapunov exponents varies continuously with $\nu_k$. 
\end{nuevo}

This theorem generalizes Theorem B of \cite{VY}, where it is proven that, for most partially hyperbolic diffeomorphism having 1-dimensional center bundle and center foliation that forms a circle bundle, ergodic measures with large entropy are hyperbolic. Observe that the hypothesis of large entropy implies the continuity of the partial entropy by \cite{LY2}.  

In \cite{BBD} it is proved the abundance of ergodic measures with vanishing center Lyapunov exponents. Comparing this result with Theorem C, we conclude that the collapse of partial entropy should be a mechanism for non-hyperbolic measures approaching hyperbolic measures. Recall that in \cite{Y} has been prove that the partial entropy varies upper-semicontinuously. 

With some small changes we are able to extend Theorem A and Theorem B to the context of partially hyperbolic volume-preserving diffeomorphisms. One of the main tools is Theorem E of \cite{ASV} which gives information only for measures in the Lebesgue class. Because of that, we are not able to prove similar versions of Theorem A and Theorem B in other settings. The situation is different for Theorem C, we demonstrate a more general result that holds for any ergodic hyperbolic measure.

\section{Preliminaries and Statements}

Let $M$ be a compact manifold and $f\colon M\to M$ a diffeomorphism. We say $f$ is \textit{partially hyperbolic} if there exist a nontrivial splitting of the tangent bundle $$TM=E^{s}\oplus E^{c}\oplus E^{u}$$ invariant under the derivative map $Df$, a Riemannian metric $\left\| \cdot \right\|$ on $M$, and positive continuous functions $\chi$, $\widehat{\chi}$, $\upsilon$, $\widehat{\upsilon}$, $\gamma$, $\widehat{\gamma}$ with 
$$\chi< \upsilon < 1 < \widehat{\upsilon}^{-1} < \widehat{\chi}^{-1} \quad  \text{and} \quad \upsilon< \gamma < \widehat{\gamma}^{-1}< \widehat{\upsilon}^{-1},$$ such that for any unit vector $v\in T_{p}M$, 
\begin{equation}\label{ph}
\begin{aligned}
\chi(p)< &\left\| Df_{p}(v) \right\|< \upsilon(p) \quad \quad \text{if} \; v\in E^{s}(p), \\
\gamma(p) < &\left\| Df_{p}(v) \right\|< \widehat{\gamma}(p)^{-1} \quad \text{if} \; v\in E^{c}(p),\\
\widehat{\upsilon}(p)^{-1}< &\left\| Df_{p}(v) \right\| < \widehat{\chi}(p)^{-1} \quad \text{if} \; v\in E^{u}(p).
\end{aligned}
\end{equation}

In the following, we mention some important consequences of partial hyperbolicity which are going to be used in this paper. We refer the reader to \cite{BDV,HPS,Sh} for more information. 
 
For every partially hyperbolic diffeomorphism the stable and unstable bundles $E^{s}$ and $E^{u}$ are uniquely integrable and their integral manifolds form two transverse (continuous) foliations $W^{s}$ and $W^{u}$, whose leaves are immersed submanifolds of the same class of differentiability as $f$. These foliations are called the \textit{strong-stable} and \textit{strong-unstable} foliations. They are invariant under $f$, in the sense that
$$f(W^{s}(x))=W^{s}(f(x)) \qquad \text{and}\qquad f(W^{u}(x))=W^{u}(f(x)),$$ where $W^{s}(x)$ and $W^{u}(x)$ denote the leaves of $W^{s}$ and $W^{u}$, respectively, passing through any $x\in M$. 

Given two points $x,y\in M$, we say $x$ is \textit{accessible} from $y$ if there exists a path that connects $x$ to $y$, which is a concatenation of finitely many subpaths, each of which lies entirely in a single leaf of $W^u$ or a single leaf of $W^s$. We call this type of path, an \textit{su-path}. 

The relation defined by $x\sim y$ if and only if $x$ is accessible from $y$ is an equivalence relation and we say that $f$ is \textit{accessible} if $M$ is the unique accessibility class. 

Although partial hyperbolicity is a $C^1$-open condition, that is, any diffeomorphism sufficiently $C^1$-close to a partially hyperbolic diffeomorphism is itself partially hyperbolic, it is not known if accessibility is a $C^1$-open condition. However, by the results in \cite{AV2}, this is true for the case of partially hyperbolic diffeomorphisms with 2-dimensional center bundle. 

\begin{defn}[$\alpha$-pinched]\label{holder} Let $f$ be a partially hyperbolic diffeomorphism and $\alpha>0$. We say that $f$ is $\alpha$-pinched if the functions in Equation (\ref{ph}) satisfy,  
\begin{equation*}
\begin{aligned}
\upsilon &< \gamma\, \chi^{\alpha} \quad \text{and} \quad \; \widehat{\upsilon} < \widehat{\gamma}\, \widehat{\chi}^{\alpha}, \\
\upsilon &< \gamma\, \widehat{\chi}^{\alpha} \quad \text{and} \quad \; \widehat{\upsilon} < \widehat{\gamma}\, \chi^{\alpha}. 
\end{aligned}
\end{equation*} 
\end{defn}

\begin{defn}[$\alpha$-bunched]\label{bunched} Let $f$ be a partially hyperbolic diffeomorphism and $\alpha>0$. We say that $f$ is $\alpha$-bunched if the functions in Equation (\ref{ph}) satisfy, 
$$\upsilon^{\alpha} < \gamma \widehat{\gamma} \qquad \text{and} \qquad \widehat{\upsilon}^{\alpha}< \gamma \widehat{\gamma}.$$ 
\end{defn}

Notice that both conditions, $\alpha$-pinched and $\alpha$-bunched, are $C^1$-open. Moreover, every partially hyperbolic system is $\alpha$-pinched for some $\alpha>0$ and if $f$ is also $C^2$, then $E^c$ is $\alpha$-H\"older. See Section 4 of \cite{PSW2}.

In the following, we give the precise definitions for the subsets of partially hyperbolic diffeomorphisms where the theorems holds. 

Let $\mathit{PH}^r(M)$ denotes the set of $C^r$ partially hyperbolic diffeomorphisms on $M$.

\begin{defn}\label{basico} Fix $N>0$. $\mathcal{B}^N(M)$ denotes the subset of $\mathit{PH}^2(M)$ where $f\in \mathcal{B}^N(M)$ if $f$ is $\alpha$-pinched and $\alpha$-bunched for some $\alpha>0$, its center bundle is 2-dimensional and $dist_{C^2}(f, id)<N.$
\end{defn}

Let $M$ be a symplectic manifold and $\omega$ denote its symplectic form, then $\mathit{PH}^r_{\omega}(M)$ denotes the set of $C^r$ partially hyperbolic diffeomorphisms preserving $\omega$.

\begin{defn}\label{symple} Fix $N>0$. $\mathcal{E}^N_{\omega}(M)$ denotes the subset of $\mathit{PH}^2_{\omega}(M)\cap \mathcal{B}^N(M)$ where $f\in \mathcal{E}^N_{\omega}(M)$ if $f$ is accessible. 
\end{defn}

By Theorem A in \cite{SW}, if $M=\mathbb{T}^{2d}$ with $d\geq 2$, then $\mathcal{E}^N_{\omega}(M)$ is non-empty. 

Let $\mu$ denote a probability measure in the Lebesgue class and $\mathit{PH}^r_{\mu}(M)$ denote the set of $C^r$ partially hyperbolic diffeomorphisms preserving $\mu$. In order to extend the main theorems to the volume-preserving setting, we need to ask for extra hypotheses.

\begin{defn}\label{pinch} Let $f$ be a partially hyperbolic diffeomorphism with $\dim E^c=2$ and $p$ a periodic point of $f$ with $n_p=per(p)$. We say that $p$ is a \emph{pinching hyperbolic periodic point} if the eigenvalues of $Df^{n_p}_p|E^c(p)$ have different norms and both norms are different from one.
\end{defn}

\begin{defn}\label{volu} Fix $N>0$. $\mathcal{E}^N_{\mu}(M)$ denotes the subset of $\mathit{PH}^2_{\mu}(M)\cap \mathcal{B}^N(M)$ where $f\in \mathcal{E}^N_{\mu}(M)$ if $f$ is accessible and has a pinching hyperbolic periodic point. 
\end{defn}

From now on, $\mu$ will denote a probability measure in the Lebesgue class and $\omega$ a symplectic form. We will use the notation $\mathcal{E}^N_{*}(M)$ with $*\in \{\mu, \omega\}$ to refer to both sets. If we are in the symplectic context, $\mu$ always denotes the Lebesgue measure associated to $\omega$.

The notion of $\alpha$-bunched defined above implies that the diffeomorphism is center bunched in the sense of Theorem 0.1 of \cite{BW}. Therefore, every diffeomorphism in $\mathcal{E}^N_{*}(M)$ is ergodic. 

Observe that for every $f\in \mathcal{E}^N_{*}(M)$ there exists a $C^1$ neighborhood, $\mathcal{W}(f)$, such that every $f\in \mathcal{W}(f)$ is partially hyperbolic with 2-dimensional center bundle, $\alpha$-pinched, $\alpha$-bunched and accessible. This last property is a consequence of the results mentioned before for partially hyperbolic diffeomorphisms with 2-dimensional center bundle. 

The Theorem of Oseledets states that for every $C^1$ diffeomorphism $f$ and every $f$-invariant probability measure $\nu$, there exists a $\nu$-full measure set $\widehat{M}$ such that for every $x\in \widehat{M}$, there exist $k(x)\in \mathbb{N}$, real numbers $\widehat{\lambda}_1(f,x)> \cdots > \widehat{\lambda}_{k(x)}(f,x)$ and a splitting $T_{x}M=E^{1}_x\oplus \cdots \oplus E^{k(x)}_x$ of the tangent bundle at $x$, all depending measurably on the point, such that
$$\lim\limits_{n\rightarrow \pm \infty} \frac{1}{n} \text{log} \left\|Df^{n}_x(v)\right\|= \widehat{\lambda}_j(f,x) \quad \text{for all} \; v\in E^j_x \setminus \{0\}. $$
The real numbers $\widehat{\lambda}_j(f,x)$ are the \emph{Lyapunov exponents} of $f$ in the point $x$. Moreover, if $(f, \nu)$ is ergodic, then the functions $k(x)$ and $\widehat{\lambda}_j(f,x)$ are constants almost everywhere. 

\begin{defn} We say that $\nu$ is \emph{hyperbolic} if the set of points with non-zero Lyapunov exponents has full $\nu$-measure. Let $\mu$ be in the Lebesgue class, if $\mu$ is hyperbolic, then we say that the diffeomorphism $f$ is \emph{non-uniformly hyperbolic}. 
\end{defn} 

If $f$ is a partially hyperbolic diffeomorphism, then we call \emph{center Lyapunov exponents} of $f$, the Lyapunov exponents associated to $E^c$. If $\dim E^c=2$ and $\nu$ is ergodic, we are going to denote them by $\lambda^c_1(f,\nu)$ and $\lambda^c_2(f,\nu)$. Observe that it is possible that $\lambda^c_1(f,\nu)=\lambda^c_2(f,\nu)$. In this case, the measure $\nu$ is hyperbolic if the center Lyapunov exponents are different form zero. 


By Lemma 2.5 of \cite{XZ}, if $f$ is a symplectic partially hyperbolic diffeomorphism, we have $$\int \log \left|\det(Df_x\vert E^c_x)\right|d\mu=0.$$ Therefore, if $f$ is also ergodic, then $\lambda^c_1(f, \mu)+\lambda^c_2(f,\mu)=0$. In this case, in order to prove that $f$ is non-uniformly hyperbolic, it is enough to conclude that the center Lyapunov exponents are different. This is the main advantage of working in the symplectic setting. 

We are now ready to give the precise statement of the results. 

\begin{teoA} For every $N>0$, the subset of non-uniformly hyperbolic diffeomorphisms in $\mathcal{E}^N_{\omega}(M)$ is $C^1$-open.
\end{teoA}

\begin{teoB} For every $N>0$, the set formed by diffeomorphisms having different center Lyapunov exponents is $C^1$-open in $\mathcal{E}^N_{\mu}(M)$.
\end{teoB} 

As already mentioned, the difference between the two theorems lies in the fact that different center Lyapunov exponents implies non-uniform hyperbolicity in the symplectic case. In Section 7, we explain why we also need to ask for the existence of a periodic point when working in the volume-preserving setting. 

The tools developed in Section 3 and 5 also allow us to prove two theorems about continuity of the center Lyapunov exponents. Recall that if we are in the symplectic context, $\mu$ always denotes the Lebesgue measure associated to $\omega$.

\begin{defn}\label{contin} Fix $N>0$. Let $*\in \{\mu, \omega\}$ and $L_{*}\colon \mathcal{E}^N_{*}(M)\to \mathbb{R}^2$ be defined by $$L_{*}(f)=(\lambda^c_1(f,\mu), \lambda^c_2(f, \mu)).$$
\end{defn}


Recall that every $f\in \mathcal{E}^N_{*}(M)$ is ergodic. 

Observe that in the symplectic case it is enough to consider the function $f\mapsto \lambda^c_1(f, \mu)$. This is again a consequence of the symmetry of the center Lyapunov exponents. 

\begin{teoC} Let $N>0$ and $r\geq 2$. There exists a subset $\mathcal{U}$ which is $C^1$-open and $C^r$-dense in $\mathcal{E}^N_{\omega}(M)$ and such that every $g\in \mathcal{U}$ is a $C^1$-continuity point of $L_{\omega}$. 
\end{teoC}

\begin{teoD} Let $N>0$ and $r\geq 2$. There exists a subset $\mathcal{U}$ which is $C^1$-open and $C^r$-dense in $\mathcal{E}^N_{\mu}(M)$ and such that every $g\in \mathcal{U}$ is a $C^1$-continuity point of $L_{\mu}$. 
\end{teoD}

Let $f\in \mathcal{B}^N(M)$, then $F=Df\vert E^c$ is a linear cocycle over $f$, that is, $F$ is a continuous transformation, $F\colon E^c\to E^c$, satisfying $\pi \circ F= f\circ \pi$ and acting by linear isomorphisms, $F_{x}\colon E^c_{x}\to E^c_{f(x)}$, on the fibers. By Furstenberg, Kesten \cite{FK}, the extremal Lyapunov exponents $$\lambda_{+}(F,x)=\lim\limits_{n\rightarrow\infty} \frac{1}{n} \text{log} \left\|F_{x}^{n}\right\| \quad \text{and} \quad \lambda_{-}(F,x)=\lim\limits_{n\rightarrow\infty} \frac{1}{n} \text{log} \left\|(F_{x}^{n})^{-1}\right\|^{-1},$$ exist at $\nu$-almost every $x\in M$, relative to any $f$-invariant probability measure $\nu$. It is clear that $\lambda_{-}(F,x)\leq \lambda_{+}(F,x)$ whenever they are defined. 

If $(f,\nu)$ is ergodic, the extremal Lyapunov exponents of $F$ are constant on a full $\nu$-measure set and we denote them by $\lambda_{+}(F, \nu)$ e $\lambda_{-}(F, \nu)$. Observe that they coincide with the center Lyapunov exponents of $(f, \nu)$.  

\begin{defn}\label{holonomy} We call \emph{invariant unstable holonomy} for $F$ a family $H^{u}$ of homeomorphisms $H^{u}_{x,y}\colon V_{x}\to V_{y}$, defined for all $x$ and $y$ in the same strong-unstable leaf of $f$ and satisfying
\begin{enumerate} [label=\emph{(\alph*)}]
\item $H^{u}_{y,z}\circ H^{u}_{x,y}= H^{u}_{x,z}$ and $H^{u}_{x,x}=Id$;
\item $F^{-1}_{y}\circ H^{u}_{x,y}= H^{u}_{f^{-1}(x),f^{-1}(y)}\circ F^{-1}_{x}$\; 
\item $(x,y, \xi)\mapsto H^{u}_{x,y}(\xi)$ is continuous when $(x,y)$ varies in the set of pairs of points in the same local strong-unstable leaf;
\item There are $C>0$ and $\eta >0$ such that $H^{u}_{x,y}$ is $(C,\eta)$-H\"older continuous for every $x$ and $y$ in the same local strong-unstable leaf.
\end{enumerate}
\emph{Invariant stable holonomy} is defined analogously, for pairs of points in the same strong-stable leaf.
\end{defn}


The \emph{projective cocycle} associated to $F$ is the continuous map $\mathbb{P}(F)\colon \mathbb{P}(E^c)\to \mathbb{P}(E^c)$ whose action on the fibers is given by the projectivization of $F_x\colon E^c_x \to E^c_{f(x)}$. Since $\dim\,E^c=2$, $\mathbb{P}(F)$ is a cocycle of circle diffeomorphisms over $f$. 

For every $f$-invariant probability measure $\nu$, there always exists a $\mathbb{P}(F)$-invariant probability measure $m$ that projects down to $\nu$. This is true because the projective cocycle $\mathbb{P}(F)$ is continuous and the domain $\mathbb{P}(E^c)$ is compact. 

The extremal Lyapunov exponents of $\mathbb{P}(F)$ for $m$ exist and satisfy,
\begin{equation}\label{lyap2}
\begin{aligned}
\lambda_{+}(\mathbb{P}(F), x,\xi)\leq \lambda_{+}(F,x)& - \lambda_{-}(F,x) \quad \text{and}\\
&\lambda_{-}(\mathbb{P}(F), x,\xi)\geq \lambda_{-}(F,x) - \lambda_{+}(F,x),
\end{aligned}
\end{equation}
for every $x\in M$ and $\xi\in \mathbb{P}(E^c_x)$ where they are defined.

We can define an invariant unstable holonomy for $\mathbb{P}(F)$ analogously to Definition \ref{holonomy}. If $H^u$ denotes a holonomy for $F$, then $h^u=\mathbb{P}(H^u)$ defines an invariant unstable holonomy for $\mathbb{P}(F)$. 

If $\pi\colon \mathcal{V}\to M$ is a fiber bundle over $M$ and $m$ a probability measure in $\mathcal{V}$ with $\pi_{*}m=\nu$, then there exists a disintegration of $m$ into conditional probability measures $\left\{m_{x} : x\in M \right\}$ along the fibers which is essentially unique, that is, a measurable family of probability measures such that $m_x(\mathcal{V}_x)=1$ for almost every $x\in M$ and $$m(U)=\int m_x(U\cap \mathcal{E}_x) d\nu(x),$$ for every measurable set $U\subset \mathcal{V}$. See \cite{Rok1}.

Let $f\in \mathcal{B}^N(M)$, $\nu$ an ergodic measure of $f$, $h^u$ an invariant unstable holonomy for $\mathbb{P}(F)$ and $m$ a probability measure on $\mathbb{P}(E^c)$. 

\begin{defn} We say that $m$ is a \emph{$u$-state} for $\nu$ if $m$ is an $\mathbb{P}(F)$-invariant measure projecting down to $\nu$ and it admits some essentially u-invariant disintegration $\left\{m_{x} : x\in M \right\}$, that is, there exists a full $\nu$-measure set $M^u$ such that $(h^{u}_{x,y})_{*}m_{x}=m_{y}$ for every $x$ and $y$ in the same strong-unstable leaf if $x,y\in M^u.$ 

The definition of \emph{s-state} for $\nu$ is analogous and we say that $m$ is an \emph{su-state} for $\nu$ if it is both an s-state and a u-state. 
\end{defn} 



\begin{defn}\label{bundlefree} 
Let $f\in \mathcal{B}^N(M)$ and $\nu$ an ergodic probability measure of $f$. We say that $f$ is \emph{bundle-free} for $\nu$ if the projective derivative cocycle $\mathbb{P}(F)$, $F=Df\vert E^c$, does not admit $su$-states for $\nu$. 
\end{defn} 

In Section 5, we prove that the subset of bundle-free diffeomorphisms on $\mathcal{E}^N_{*}(M)$ is a $C^1$-open set. Moreover, if $M=\mathbb{T}^d$ with $d\geq 2$, the subset of bundle-free diffeomorphisms on $\mathcal{E}^N_{\omega}(M)$ is non-empty. We can also set some hypotheses for diffeomorphisms on $\mathcal{E}^N_{*}(M)$ in order to guarantee that it is $C^r$-approximated for bundle-free diffeomorphisms for every $r\geq 2$. For example, if $f\in \mathcal{E}^N_{\omega}(M)$ has a quasi-elliptic periodic point, it verifies this property, see Section 7.2.2 of \cite{M} and Remark 2.9 of \cite{ASV}. 

The precise statement of Theorem C is the following: 

\begin{nuevo} Let $f\in \mathcal{E}^N(M)_{\omega}(M)$ be a non-uniform hyperbolic bundle-free diffeomorphism. Then, for every sequence of ergodic measures $\nu_k$ such that $\nu_k\to \mu$ in the weak$^{*}$ topology and for $i=s$ or $i=u$ verifies
\begin{equation*}
\lim_{k\to \infty} h_{\nu_k}(f,W^i)=h_{\mu}(f,W^i),
\end{equation*}
there exists $K\in \mathbb{N}$ such that $\nu_k$ is a hyperbolic measures of $f$ for every $k\geq K$. 
Moreover, for $j=1,2$, $\lambda^c_j(f, \nu_k)\to \lambda^c_j(f, \mu)$ when $k\to \infty$. 
\end{nuevo}

Next, we state a more general result which contains Theorem C. 

\begin{nuevo2} Let $f\in \mathcal{B}^N(M)$, $\nu$ an ergodic hyperbolic probability measure for $f$ and $f$ bundle-free for $\nu$. Then, for every sequence of diffeomorphisms $f_k\in \mathcal{B}^N(M)$ and every sequence $\nu_k$ of ergodic measures of $f_k$ such that $f_k\to f$ in the $C^1$-topology, $\nu_k\to \nu$ in the weak$^{*}$ topology and for $i=s$ or $i=u$ verifies
\begin{equation*}
\lim_{k\to \infty} h_{\nu_k}(f_k,W^i_k)=h_{\nu}(f,W^i),
\end{equation*}
there exists $K\in \mathbb{N}$ such that $\nu_k$ is a hyperbolic measures of $f_k$ for every $k\geq K$. 
Moreover, for $j=1,2$, $\lambda^c_j(f_k, \nu_k)\to \lambda^c_j(f, \nu)$ when $k\to \infty$. 
\end{nuevo2}

\subsection*{Strategy of the proofs}

Several results about Lyapunov exponents have been proved for linear cocycles, see for instance \cite{ASV, BBB, BV2}. Here we apply those ideas to the derivative cocycle $F=Df\vert E^c$. The key difference is that we are perturbing the cocycle and the diffeomorphism in the base at the same time. 

The main objective for the proof of Theorem A and Theorem B is to get a characterization of the discontinuity points of the function $L_{\omega}$ in Definition \ref{contin}. We are able to conclude this result in a more general context: we consider the function depending on the diffeomorphism and on the measure. This is done in Section 6 using the previous results of Section 3 and 5. 

The first step is to prove that the holonomies of the cocycle $F=Df\vert E^c$ depends continuously on $f\in \mathcal{B}^N(M)$. This is obtained in Proposition \ref{continuity} and Proposition \ref{contF}. Section 5 is devoted to prove that the limit of $u$-states is a $u$-state.

Corollary \ref{discont} states that if $f\in \mathcal{E}^N_{\omega}(M)$ is a discontinuity point of $L_{\omega}$, then the fiber bundle $\mathbb{P}(E^c(f))$ admits two continuous sections, $x\mapsto a_x$ and $x\mapsto b_x$. Moreover, if $\mathit{F=Df\vert E^c}$, then those sections are invariant by the cocycle $\mathbb{P}(F)$ and by the invariant stable and unstable holonomies of $\mathbb{P}(F)$.

In order to conclude the proof of Theorem A, we use the argument in Section 6 of \cite{LMY}. We prove that the diffeomorphisms having zero center Lyapunov exponents form a closed subset. The proof is by contradiction: suppose there exist a sequence $f_k$ and a diffeomorphism $f$ such that $f_k\to f$ in the $C^1$ topology, $\lambda^c_1(f_k)=\lambda^c_2(f_k)=$ for every $k\in \mathbb{N}$ and $\lambda^c_1(f)\neq \lambda^c_2(f)$. In particular, $f$ is a discontinuity point of $L_{\omega}$. Therefore, by the result above, we have two continuous sections of $\mathit{\mathbb{P}(E^c(f))}$, $x\mapsto a_x$ and $x\mapsto b_x$. 

Since the center Lyapunov exponents of $f_k$ are zero, we use the Invariance Principle (Theorem B of \cite{ASV}) to prove that for every $k$ big enough, there exists a continuous section of $\mathit{\mathbb{P}(E^c(f_k))}$, $x\mapsto a_{k,x}$.

By the continuity of the holonomies, we conclude that the section $a_{k,x}$ has to be close to $a_x$ or to $b_x$ for every $x\in M$. However, this implies that $\lambda^c_1(f_k)\to \lambda^c_1(f)$ or $\lambda^c_2(f_k)\to \lambda^c_2(f)$ which contradicts our assumption. 

The proof of Theorem B is a combination of the results in Section 5, the arguments in the proof of Theorem A and the results of \cite{M}. We find a diffeomorphism $g$ which is $C^r$-arbitrarily close to $f$ and it is a continuity point for $L_{\omega}$. More precisely, $g$ has the following property: let $G=Dg\vert E^c(g)$, then there is no continuous section invariant by $\mathbb{P}(G)$ and its holonomies. As a consequence of Section 5, this property is $C^1$-open and then we conclude Theorem B. 

The proofs of Theorem A' and B' are analogous. 

Theorem C and Theorem D are a direct consequence of the characterization of discontinuity points established in Section 6. 

\section{Continuity of holonomies}

The main result in this section is Proposition \ref{contF}. This is a key property of diffeomorphisms in $\mathcal{B}^N(M)$ and is going to be used in Section 5. It states that the unstable holonomies varies continuously with the diffeomorphism.


Consider the \emph{center derivative cocycle} associated to $f$, $F=Df\vert E^c$. 

Since $M$ is compact, we can define a distance in $TM$ in the following way: For every $x,y\in M$ close enough, denote $\pi_{x,y}:T_{x}M\longrightarrow T_{y}M$ the parallel transport along $\zeta$, where $\zeta$ is the geodesic satisfying $dist(x,y)=\text{length}(\zeta)$. Then, given two points $(x,v)$ and $(y,w)$ in $TM$ define $$d((x,v),(y,w))=dist(x,y) + \left\|\pi_{x,y}(v)-w\right\|.$$

To simplify the notation we are going to write $$d((x,v), (y,w))=d(v,w)\quad \text{and} \quad \pi^n_{x,y}=\pi_{f^{-n}(x),f^{-n}(y)}.$$

If $x,y\in M$ with $y\in W^u_{loc}(x)$, then for every $n\in \mathbb{N}$ define $$A_n(f,x,y)=(F^{-n}(y))^{-1}\circ P_{E^c(f^{-n}(y))} \circ \pi^n_{x,y} \circ F^{-n}(x),$$ and $$A_0(f,x,y)= P_{E^c(y)} \circ \pi_{x,y}\vert E^c(x).$$

Here $$F^{-n}(x)=F^{-1}(f^{-n+1}(x))\circ \dots \circ F^{-1}(x),$$ and $P_{E^c(x)}$ denotes the orthogonal projection over $E^c(x)$.

In Section 3 of \cite{M}, has been proved that if $f$ is a $C^2$ partially hyperbolic diffeomorphism which is $\alpha$-pinched and $\alpha$-bunched for some $\alpha>0$, then the sequence $A_n(f,x,y)$ is a Cauchy sequence and the limit $$H^u_{x,y}=\lim\limits_{n\rightarrow \infty} A_n(f,x,y),$$ defines an invariant unstable holonomy for $F=Df\vert E^c$. 

The goal of this subsection is to prove that these holonomies varies continuously in $\mathcal{B}^N(M)$ with the $C^1$ topology, see Proposition \ref{continuity} below. In order to do so, we give an outline of the proof in \cite{M}. 

Since $f$ is $C^2$ and $\alpha$-pinched, there exits $C_1(f)>0$ such that if $x,y$ are close enough, then 
\begin{equation}\label{holderE}
d(E^c_x, E^c_y)<C_1(f)\,dist(x,y)^{\alpha}.
\end{equation}
Here we use again the parallel transport $\pi_{x,y}\colon T_xM\to T_yM$ to compare subspaces of different tangent spaces. Also, there exists $C_2(f)>0$ such that for every $(x,v),\, (y,w)\in TM$, 
\begin{equation}\label{C2}
d(Df(x,v), Df(y,w))< C_2(f)\, d(v,w).
\end{equation}

Moreover, the $\alpha$-bunched condition (Definition \ref{bunched}) implies there exists $\varsigma<1$ such that $$\upsilon^{\alpha \varsigma} < \gamma \widehat{\gamma} \qquad \text{and} \qquad \widehat{\upsilon}^{\alpha \varsigma}< \gamma \widehat{\gamma}.$$ 

The three estimations above allow us to conclude that 
\begin{equation}\label{Cauchy}
\left\|A_{n+1}(f,x,y)- A_n(f,x,y)\right\|\leq C_1(f)\, C_2(f)\, \widehat{\upsilon}^n(x)^{(1-\varsigma)\alpha}dist(x,y)^{\alpha}.
\end{equation}

This proves that the sequence $A_n(f,x,y)$ is a Cauchy sequence and therefore the limit exists. 

The following proposition states that the invariant unstable holonomy for $F=Df\vert E^c$ varies continuously with $f$. 

\begin{prop}\label{continuity} Let $N>0$. Fix $f\in \mathcal{B}^N(M)$, $x\in M$ and $a\in E^c(x,f)$. For every $\epsilon>0$ there exist $\delta>0$ and a neighborhood of $f$ in the $C^1$ topology, $\mathcal{V}(f)$, such that for every $y\in W^u_{loc}(x,f)$, every $g\in \mathcal{B}^N(M)\cap \mathcal{V}(f)$, every $w,z\in M$ with $w\in W^u_{loc}(z,g)$, $dist(x,z)<\delta$ and $dist(y,w)<\delta$ and every $b\in E^c(z,g)$ with $d(a,b)<\delta$, we have
$$d(H^u_{x,y}(f)(a), H^u_{z,w}(g)(b))< \epsilon.$$
\end{prop}
\begin{proof}

First observe that if $f$ is $\alpha$-pinched and $\alpha$-bunched, then there exists a $C^1$ neighborhood of $f$ where every $g$ is also $\alpha$-pinched and $\alpha$-bunched with the same $\alpha$.

By Equation (\ref{Cauchy}) applied to $f$ and $g$, we have for every $n\in \mathbb{N}$,

$$\left\|H^u_{x,y}(f)- A_n(f,x,y)\right\|\leq C_1(f)\, C_2(f)\, \widehat{\upsilon}^n(x)^{(1-\varsigma)\alpha}dist(x,y)^{\alpha}.$$
                                  
$$\left\|H^u_{z,w}(g)- A_n(g,z,w)\right\|\leq C_1(g)\, C_2(g)\, \widehat{\upsilon}^n(z)^{(1-\varsigma)\alpha}dist(z,w)^{\alpha}.$$ 

Here $C_1$ is defined by Equation (\ref{holderE}) and can be taken uniform in a $C^1$ neighborhood of $f$. See, for example, \cite{W}. We can also take $C_2$ in Equation (\ref{C2}) uniform since $f,g\in \mathcal{B}^N(M)$. 

Then, the proposition is a consequence of the continuity of $$A_n(f,x,y)=(F^{-n}(y))^{-1}\circ P_{E^c(f^{-n}(y))} \circ \pi^n_{x,y} \circ F^{-n}(x),$$ as a function of $(f,x,y)$.

More precisely, the distance $$d(A_n(f,x,y)(a), A_n(g,z,w)(b))$$ can be bounded by an expression that depends on the following terms:
$$dist(x,z), \quad dist(y,w), \quad d(a,b),$$
$$\angle (E^c(x,f), E^c(x,g)), \quad \angle (E^c(y,f), E^c(y,g)),$$ 
$$\left\|Df^{-1}(f^{-j}(x))-Dg^{-1}(f^{-j}(z))\right\|\quad \text{and} \quad \left\|Df(f^{-j}(y))-Dg(f^{-j}(w))\right\|,$$ for $j\in\{0,...,n\}.$
\end{proof}

\begin{obs} Observe that $\delta$ in Proposition \ref{continuity} is independent of the point $y\in W^u_{loc}(x,f)$.
\end{obs} 

\subsection{Projective Cocycle}

Now, we consider the projective cocycle associated to $f\in \mathcal{B}^N(M)$, that is $\mathbb{P}(F)$ with $F=Df\vert E^c$. We have the following results which is a consequence of Proposition 3.10 of \cite{ASV}. Here, $h^u=\mathbb{P}(H^u)$ and $H^u$ is the unstable holonomy constructed above.

\begin{prop}\label{Fu} Fix $N>0$. Let $f\in \mathcal{B}^N(M)$ and $F=Df\vert E^c$. 

For every $(x,v)\in \mathbb{P}(E^c)$ define $$\mathcal{F}^u(x,v)=\{(x,h^u_{x,y}(v))\colon y\in W^u(x)\}.$$ Then, $\mathcal{F}^u=\{\mathcal{F}^u(x,v)\colon (x,v)\in \mathbb{P}(E^c)\}$ is a partition of $\mathbb{P}(E^c)$ that satisfies the following: 
\begin{enumerate} [label=\emph{(\alph*)}]
\item There exists $C>0$ such that every $\mathcal{F}^u(x,v)$ is a $(C,\alpha)$-Holder continuous graph over $W^u(x)$.
\item The partition is invariant: $$\mathit{\mathbb{P}(F^{-1})(\mathcal{F}^u(x,v))\subset \mathcal{F}^u(\mathbb{P}(F^{-1})(x,v))} \text{ for all }(x,v)\in \mathbb{P}(E^c).$$
\item If $\mathcal{F}^u_{loc}(x,v)=\{(x,h^u_{x,y}(v))\colon y\in W^u_{loc}(x)\}$ and $(y,w)\in \mathcal{F}^u_{loc}(x,v)$, then $$\mathit{d(\mathbb{P}(F^{-n})(x,v),\mathbb{P}(F^{-n})(y,w))}\to 0 \text{ exponentially fast.}$$ 
\end{enumerate}
\end{prop}

In item (c) we are considering the following distance to compare points in different fibers: For every $x,y\in M$ close enough, denote $\pi_{x,y}:T_{x}M\longrightarrow T_{y}M$ the parallel transport along $\zeta$, where $\zeta$ is the geodesic satisfying $dist(x,y)=\text{length}(\zeta)$. Then, given two points $(x,v)$ and $(y,w)$ in $\mathit{\mathbb{P}(TM)}$ define $$d((x,v),(y,w))=dist(x,y) + \angle(\pi_{x,y}(v), w).$$

The following proposition is a consequence of Proposition \ref{continuity} written in the notation of Proposition \ref{Fu}. 

\begin{prop}\label{contF} Fix $N>0$. Suppose $f_k, f\in \mathcal{B}^N(M)$ and $f_k\to f$ in the $C^1$ topology. Then, for every $x\in M$ and every sequence $\{v_k\}_{k\in \mathbb{N}}$ such that $v_k\in \mathit{\mathbb{P}(E^c(x,f_k))}$ for every $k\in \mathbb{N}$ and $v_k\to v$ when $k\to \infty$, we have 
$$dist_{H}\left(\mathcal{F}^{u,k}_{loc}(x,v_k), \mathcal{F}^{u}_{loc}(x,v)\right)\to 0 \text{ when } k\to \infty.$$
Here $dist_H$ denotes the Hausdorff distance for subsets of $\mathit{\mathbb{P}(TM)}$.
\end{prop}

\section{u-states}

In this following we will always be considering $N>0$, $f\in \mathcal{B}^N(M)$ and $\nu$ an ergodic probability measure of $f$. Let $F=DF\vert E^c$ denote the center derivative cocycle associated to $f$. 

The next classical result provides a relation between the extremal Lyapunov exponents of $F$ and the invariant probability measures of its projectivization. More precisely, 

\begin{lemma}\label{fust} Define $\Phi(x,v)=\log \left\|F_x(v)\right\|$ for every $(x,v)\in \mathbb{P}(E^c)$. Then, the exponent $\lambda_{+}(F, \nu)$ coincides with the maximum of $\int{\Phi(x,v)\, dm}$ over all $\mathbb{P}(F)$-invariant probability measures $m$ projecting down to $\nu$. 
\end{lemma}

This lemma shows that in order to obtain properties of $\lambda_{+}(F, \nu)$, it is natural to study the $\mathbb{P}$(F)-invariant measures. One main example in this direction is the Invariance Principle introduced by Ledrappier \cite{L} and generalized by Avila and Viana \cite{AV1} and Avila, Santamaria and Viana \cite{ASV}. Roughly speaking, it states that if $\lambda_{+}(F,\nu)=\lambda_{-}(F, \nu)$, then every $\mathbb{P}(F)$-invariant probability measure is an $su$-state. 

More recently, Tahzibi and Yang \cite{TY} introduced a new criterion for a $\mathbb{P}(F)$-invariant measure to be a $u$-state. This allow them to give another proof of the Invariance Principle in the case that the base is an Anosov diffeomorphism and the linear cocycle is differentiable. 

In the following, we give the precise definitions and state the result of \cite{TY}. In the next section, we use this criterion to prove that the limit of $u$-states is a $u$-state. 


\subsection{Entropy along expanding foliations}

In order to state the criterion for $u$-states, we need to consider the definition of entropy along an expanding foliation. We are going to recall some basic definitions for the entropy of measurable partitions, we refer the reader to \cite{Rok1, Rok2} for more information. 

\begin{defn} A partition $\xi$ of $M$ is called \emph{measurable} if there is a sequence of finite partitions $\{\xi_n\}_{n\in \mathbb{N}}$ such that $\xi=\bigvee_{n\in \mathbb{N}} \xi_n$ and the elements of $\xi_n$ are measurable up to $\nu$-zero measure. 
\end{defn}

Given a partition $\xi$ of $M$ and $x\in M$, we denote $\xi(x)$ the element of the partition which contains the point $x$. Given $\xi$ and $\eta$ two measurable partitions, we say that $\xi<\eta$ if for $\nu$-almost every $x\in M$, $\eta(x)\subset \xi(x)$.

If $\xi$ is a measurable partition, then there exists a disintegration of $\nu$ into conditional probabilities $\left\{\nu^{\xi}_{x} : x\in M \right\}$ which is essentially unique, that is, a measurable family of probability measures such that $\nu^{\xi}_{x}(\xi(x))=1$ for almost every $x\in M$ and $$\nu(U)=\int \nu^{\xi}_{x}(U\cap \xi(x)) d\nu,$$ for every measurable set $U\subset M$. 

Let $\xi$ and $\eta$ be two measurable partitions, then for every element $B$ of $\eta$, $\xi$ induces a partition $\xi_B$ of $B$. Denote $\nu^{\eta}_{B}$ the element of the disintegration of $\nu$ relative to $\eta$ supported on $B$. The \emph{mean conditional entropy} of $\xi$ relative to $\eta$ is defined by, 
$$H_{\nu}(\xi\vert \eta)=-\int_M{\log \nu^{\eta}_{B(x)}(\xi_{B(x)}(x)) \, d\nu}.$$  
We observe that this number can be infinite. 

If the measurable partitions $\{\eta_n\}_{n\in \mathbb{N}}$ and $\eta$ verify that $\eta_n < \eta_{n+1}$ for every $n\in\mathbb{N}$ and $\eta=\bigvee_{n=1}^{\infty} \eta_n$, then we write $\eta_n\nearrow \eta$. 

The following result will be useful in the next sections. 

\begin{lemma}\label{limpart} Let $\{\eta_n\}_{n\in \mathbb{N}}$, $\eta$ and $\xi$ be measurable partitions such that $\eta_n\nearrow \eta$ and $H_{\nu}(\xi\vert \eta_1)<\infty$, then $H_{\nu}(\xi\vert \eta_n)\searrow_{n\to \infty} H_{\mu}(\xi\vert \eta)$.
\end{lemma} 

Recall that we are considering $N>0$, $f\in \mathcal{B}^N(M)$ and $\nu$ an $f$-invariant probability measure.

\begin{defn} The entropy of $f$ relative to a measurable partition $\xi$ is given by $$h_{\nu}(f,\xi)=H_{\nu}(\xi\vert \xi^{+}),$$
where $\xi^{+}=\bigvee_{n=0}^{\infty} f^n\xi$. 
\end{defn}

\begin{defn} A measurable partition is said to be \emph{increasing} if $f\xi<\xi$. In this case, $h_{\nu}(f,\xi)=H_{\nu}(\xi\vert f\xi)$. 
\end{defn}

\begin{defn}\label{subo} We say a measurable partition $\xi$ is $\nu$-subordinated to the foliation $W^u$ if for $\nu$-almost every $x\in M$:
\begin{enumerate} [label=\emph{(\alph*)}]
\item $\xi(x)\subset W^u(x)$ and $\xi(x)$ has uniformly small diameter inside $W^u(x)$,
\item $\xi(x)$ contains an open neighborhood of $x$ inside the leaf $W^u(x)$ and 
\item $f\xi<\xi$ (increasing partition). 
\end{enumerate}
\end{defn}

By Lemma 3.1.2 of \cite{LY1}, given $\xi_1$ and $\xi_2$, two measurable partitions which are $\nu$-subordinated to the foliation $W^u$, we have $h_{\nu}(f,\xi_1)=h_{\nu}(f,\xi_2)$ and this quantity is finite. Therefore, we can define the entropy of $f$ along the unstable foliation by $h_{\nu}(f,W^u)=h_{\nu}(f,\xi)$ for some measurable partition $\xi$ which is $\nu$-subordinated to $W^u$. 

Analogously, if $F=Df\vert E^c$, $m$ is a $\mathbb{P}(F)$-invariant probability measure and $\mathcal{F}^u$ is given by Proposition \ref{Fu}, we can define $h_{m}(\mathbb{P}(F),\mathcal{F}^u)$.  

The following result is a consequence of Section 4 and 5 of \cite{TY}. They proved the next criterion in the case that the base is a $C^2$ Anosov diffeomorphism. However, their proof can be adapted to this setting without major changes. See also Section 6 of \cite{VY}.  

\begin{theorem}\label{criterium} 
Fix $N>0$, $f\in \mathcal{B}^N(M)$ and $\nu$ an $f$-invariant probability measure. Let $F=Df\vert E^c$, $\mathcal{F}^u$ given by Proposition \ref{Fu} and $m$ a $\mathbb{P}(F)$-invariant probability measure projecting down to $\nu$. Then, 
$$h_{m}(\mathbb{P}(F), \mathcal{F}^u)\leq h_{\nu}(f,W^u),$$ and equality holds if and only if $m$ is a $u$-state. 
\end{theorem} 

\section{Limit of $u$-states}

The main result of this section is the next proposition which states that the limit of $u$-states is a $u$-state. This has already been proved in several contexts, for example, see Proposition 5.17 of \cite{V} for a proof for locally constant cocycles and Lemma 4.3 of \cite{BBB} and Corollary 2.3 of \cite{TY} for linear cocycles over hyperbolic maps. For linear cocycles over partially hyperbolic diffeomorphisms, it has been stated in Corollary 5.3 of \cite{AV1} and a detailed proof can be found in Appendix A of \cite{Mau}. For the case of the derivative cocycle and the Lebegue measure, the classical argument has been recently adapted in \cite{OP}. In order to be able to extend the proof to general measures $\nu_k$, Remark 8.6 of \cite{OP} establishes conditions on the pullback of $\nu_k$ by foliation charts. The hypothesis on the measures in our proposition is more clear from a dynamical point of view. 

Although we chose to present our result for the derivative cocycle of diffeomorphisms in $\mathcal{B}^N(M)$, the argument also works in the easier case of linear cocycles with a fixed base giving another proof of the known results. The hypothesis of $\dim E^c=2$ can also be removed. 

\begin{prop}\label{limitugral}
Fix $N>0$. Let $f_k,f\in \mathcal{B}^N(M)$ for every $k\in \mathbb{N}$, $f_k\to f$ in the $C^1$ topology, $F=Df\vert E^c$ and $F_k=Df_k\vert E^c(f_k)$. 

Suppose $\nu_k$ is an $f_k$-invariant probability measure for every $k\in \mathbb{N}$ and $\nu$ is an $f$-invariant probability measure such that $\nu_k\to \nu$ in the weak$^*$ topology and 
\begin{equation*}
\lim_{k\to \infty} h_{\nu_k}(f_k,W^u_k)=h_{\nu}(f,W^u)
\end{equation*}

If $m_k$ is a $u$-state for $\mathbb{P}(F_k)$ and $\nu_k$ and $m_k\to m$ in the weak$^{*}$ topology, then $m$ is a $u$-state for $\mathbb{P}(F)$ and $\nu$.  
\end{prop} 

First, we observe that since $m_k$ is a $\mathbb{P}(F_k)$-invariant probability measure projecting down to $\nu_k$ for every $k\in \mathbb{N}$, then the limit measure $m$ satisfies the following: 

\begin{enumerate}[label=(\alph*)]
\item $\mathit{supp}\; m \subset \mathbb{P}(E^c(f))$,
\item $m$ projects down to $\nu$ and
\item $m$ is $\mathbb{P}(F)$-invariant.
\end{enumerate}

Moreover, since $m_k$ is a $u$-state for $\mathbb{P}(F_k)$, by Theorem \ref{criterium}, we have for every $k\in \mathbb{N}$,
\begin{equation}\label{Parte1}
h_{m_k}(\mathbb{P}(F_k), \mathcal{F}^{u}_k)=h_{\nu_k}(f_k, W^{u}_k).
\end{equation}

In order to conclude Proposition \ref{limitugral}, it is enough to prove: 

\begin{prop}\label{limitsup}
Fix $N>0$. Let $f_k,f\in \mathcal{B}^N(M)$ for every $k\in \mathbb{N}$, $f_k\to f$ in the $C^1$ topology, $F=Df\vert E^c$ and $F_k=Df_k\vert E^c(f_k)$. 

Suppose $\nu_k$ is an $f_k$-invariant probability measure for every $k\in \mathbb{N}$ and $\nu$ is an $f$-invariant probability measure such that $\nu_k\to \nu$ in the weak$^*$ topology.

If $m_k$ is a $\mathbb{P}(F_k)$-invariant probability measure projecting down to $\nu_k$ and $m_k\to m$ in the weak$^{*}$ topology, then
\begin{equation*}
\limsup_{k\to \infty} h_{m_k}(\mathbb{P}(F_k), \mathcal{F}^{u}_k)\leq h_{m}(\mathbb{P}(F), \mathcal{F}^{u}).
\end{equation*}
\end{prop} 

\begin{obs} We remark that the above proposition holds for every sequence of measures $m_k$ which are $\mathbb{P}(F_k)$-invariant probability measures, not necessarily being $u$-states. 
\end{obs} 

If the above proposition is true, then using Equation (\ref{Parte1}), we will have $$\limsup_{k\to \infty} h_{\nu_k}(f_k, W^{u}_k)\leq h_{m}(\mathbb{P}(F), \mathcal{F}^{u}).$$ Then, the hypotheses on $h_{\nu_k}(f_k, W^u_k)$, $$h_{\nu}(f, W^{u})\leq h_{m}(\mathbb{P}(F), \mathcal{F}^{u}).$$ This combined with Theorem \ref{criterium} gives $$h_{\nu}(f, W^{u})=h_{m}(\mathbb{P}(F), \mathcal{F}^{u}),$$ and therefore, $m$ will be a $u$-state. 

Observe that since $f$ is $C^2$, if $\nu_k=\nu=\mu$ where $\mu$ is a measure in the Lebesgue class, we have the following identity due to Theorem C' of \cite{LY2}: $$h_{\mu}(f_k, W^{u}_k)=\int_M{\log \left|\det Df_k\vert E^u(f_k)\right| d\mu},$$ for every $k\in \mathbb{N}$ and the same formula holds for $f$. 

Therefore, since $f_k\to f $ in the $C^1$ topology, 
\begin{equation}\label{Parte2}
\lim_{k\to \infty} h_{\mu}(f_k, W^{u}_k)=h_{\mu}(f, W^{u}).
\end{equation}

This observation together with Proposition \ref{limitugral} implies the following corollary. 

\begin{cor}\label{limitu}
Fix $N>0$. Let $*\in \{\mu, \omega\}$, $f_k, f\in \mathcal{E}^N_{*}(M)$ for every $k\in \mathbb{N}$, $f_k\to f$ in the $C^1$ topology, $F=Df\vert E^c$ and $F_k=Df_k\vert E^c(f_k)$. 

If $m_k$ is a $u$-state for $\mathbb{P}(F_k)$ projecting down to $\mu$ and $m_k\to m$ in the weak$^{*}$ topology, then $m$ is a $u$-state for $\mathbb{P}(F)$ and $\mu$.  
\end{cor}

The remaining of this Section is dedicated to prove Proposition \ref{limitsup}. As already mentioned, this proposition has been proved in \cite{Y} for partially hyperbolic diffeomorphisms. Since the map $\mathbb{P}(F)$ is not differentiable, in fact, it is only H\"older continuous, we can not apply that result to our setting. However, we are going to adapt those arguments. 

\subsection*{Strategy of the proof of Proposition \ref{limitsup}}
The first step in the proof is to construct subordinated measurable partitions to calculate the entropies $h_{m}(\mathbb{P}(F), \mathcal{F}^{u})$ and $h_{m_k}(\mathbb{P}(F_k), \mathcal{F}^{u}_k)$. This is done in sections 5.1 and 5.2. Since $f_k\to f$, we are able to construct these partition in a uniform way. 

More precisely, we define a finite partition of $\mathbb{P}(E^c)$, denoted by $\widehat{\mathcal{A}}$ and consider $\widehat{\mathcal{A}}^{u}$ as the partition defined by the intersection of elements of $\widehat{\mathcal{A}}$ and leaves of the foliation $\mathcal{F}^u$. Then, we prove that $$\widehat{\xi}^u=\bigvee_{j=0}^{\infty} \mathbb{P}(F^j)(\widehat{\mathcal{A}}^{u}),$$ is a measurable partition $m$-subordinated to $\mathcal{F}^u$. Therefore, 
\begin{equation}\label{nueva}
h_{m}(\mathbb{P}(F), \mathcal{F}^{u})=H_m\left(\widehat{\mathcal{A}}^u \Big| \bigvee_{j=1}^{\infty} \mathbb{P}(F^{j})(\widehat{\mathcal{A}}^u) \right).
\end{equation} 

The second step is Proposition \ref{infn}. It states that $h_{m}(\mathbb{P}(F), \mathcal{F}^{u})$ can be approximated by $H_m\left(\bigvee_{j=0}^n \mathbb{P}(F^{-j})(\widehat{\mathcal{A}}^u)\Big| \widehat{\mathcal{A}}^u\right)$. Observe that in Equation (\ref{nueva}), we are considering an infinite number of iterates and the proposition allow us to reduce the problem to consider a finite number of iterates. 

In section 5.3, for every $n\in \mathbb{N}\cup \{0\}$, we consider a sequence of finite partitions, denoted by $\widehat{\mathcal{A}}^n_l$ that verifies, $$\widehat{\mathcal{A}}^n_l\nearrow_{l\to \infty} \bigvee_{j=0}^n \mathbb{P}(F^{-j})(\widehat{\mathcal{A}}^u).$$ Proposition \ref{relac} shows that this property and the definition of the partitions $\widehat{\mathcal{A}}^n_l$ imply that $H_m(\widehat{\mathcal{A}}^n_l\vert \widehat{\mathcal{A}}^0_l)$ approximates $H_m\left(\bigvee_{j=0}^n \mathbb{P}(F^{-j})(\widehat{\mathcal{A}}^u)\Big| \widehat{\mathcal{A}}^u\right).$ Observe that we reduce the problem once more. Combining the two results, we are able to approximate $h_{m}(\mathbb{P}(F), \mathcal{F}^{u})$ by the conditional entropy of two finite partitions. 

All the above results are also true for $h_{m_k}(\mathbb{P}(F_k), \mathcal{F}^{u}_k)$. Therefore, in order to conclude Proposition \ref{limitsup} it is enough to study the continuity properties of the partitions $\widehat{\mathcal{A}}^n_l$ and $\widehat{\mathcal{A}}^n_{k,l}$. This is a consequence of Proposition \ref{contF} and the fact that the partitions were constructed in a uniform way.


\subsection{Preliminaries.} 

If $f$ is a $C^1$ partially hyperbolic diffeomorphism of a manifold $M$ of dimension $d$, we say that $(B, \Phi, D)$ is a $W^u$-foliation box if: 
\vspace{0.2cm}
\begin{enumerate}[label=(\alph*)]
\item $\Phi\colon [0,1]^{d-u}\times [0,1]^{u}\to M$ is a topological embedding,
\vspace{0.2cm}
\item every plaque $P_x=\Phi({x}\times [0,1]^{u})$ is contained in a leaf of $W^u$,
\vspace{0.2cm}
\item every $\Phi(x,\cdot)\colon [0,1]^u\to M$ is a $C^1$ embedding depending continuously with $x$ in the $C^1$ topology, 
\vspace{0.2cm}
\item $D=\Phi([0,1]^{d-u}\times {0})$ and $B$ is the image of $\Phi$.
\end{enumerate}
\vspace{0.2cm}

Since $M$ is compact, we can fix a finite cover consisting of foliation boxes $\{(B_i, \Phi_i, D_i)\}_{i=1}^q$. 
\vspace{0.1cm}

If $f_k\to f$ in the $C^1$ topology, then for every $k\in \mathbb{N}$ it is possible to take a finite cover of $M$ by $W^u_k$ foliation boxes $\{(B^k_i, \Phi^k_i, D_i)\}_{i=1}^q$ such that for every $i\in \{1,...,q\}$:

\begin{enumerate}[label=(\alph*)]
\item $\Phi^k_i\to \Phi_i$ in the $C^0$ topology, 
\vspace{0.2cm}
\item for every $x\in D_i$, $\Phi^k_i(x,\cdot)\to \Phi_i(x,\cdot)$ in the $C^1$ topology.
\end{enumerate}
\vspace{0.2cm}

Observe that the cross-sections $D_i$ are the same for every foliation box independent of $k$. 

Fix $N>0$. Let $f_k,f\in \mathcal{B}^N(M)$ such that $f_k\to f$ in the $C^1$ topology. Fix finite covers like above associated to $f$ and $f_k$. 

Take $r_0<1$ to be a Lebesgue number of the open cover $\{B_i\}_{i=1}^q$. Define $$I\colon M \to \{1,...,q\} \text{ such that } B(x, r_0)\subset B_{I(x)},$$ for every $x\in M$. Since $f_k\to f$ in the $C^1$ topology, we can suppose that $$B(x,r_0)\subset B^k_{I(x)} \text{ for every } k\in \mathbb{N}.$$ 

If $\widehat{\upsilon}$ is the function in Equation (\ref{ph}) that has be taken uniform for $f$ and $f_k$, we define $\zeta=\max_{x\in M} \widehat{\upsilon}$. Fix $\zeta_0$ such that $\zeta<\zeta_0<1$. 

\subsection{Construction of a subordinated partition.}

Let $\nu_k$ be an $f_k$-invariant probability measure for $f_k$ for every $k\in \mathbb{N}$ and $\nu_k\to \nu$ in the weak$^*$ topology. 

The first step is to construct a measurable partitions $\nu_k$-subordinated to $W^u_k$ that are uniform (in some sense) for $f$ and $f_k$. This type of construction has already appear at \cite{LS,Y}.

By Proposition 3.1 of \cite{Y}, we can construct a finite partition of $M$, denoted by $\mathcal{A}$, such that every element of $\mathcal{A}$ has diameter smaller that $r_0$ and there exist constants $C_0, C_k>0$ such that for every $i\geq 0$,
\begin{equation}\label{frontera}
\nu_k(B(\partial \mathcal{A}, \zeta^i))\leq C_k \zeta_0^i\quad  \text{and}\quad \nu(B(\partial \mathcal{A}, \zeta^i))\leq C_0 \zeta_0^i.
\end{equation}

Define the partition $\mathcal{A}^{u}$ such that every element is the intersection between an element of $\mathcal{A}$ and a leaf of $W^u$. 

Consider $\xi^u=\bigvee_{j=0}^{\infty} f^j(\mathcal{A}^{u})$. Then, it is clear that $\xi^u$ satisfies conditions (a) and (c) in Definition \ref{subo}. In order to see that condition (b) is verified, we used Equation (\ref{frontera}) and the fact that the action of $f^{-1}$ in $W^u$ is controlled by $\zeta$ (See Lemma 3.2 of \cite{Y}). Therefore, $\xi^u$ is a measurable partition $\nu$-subordinated to $W^u$. 

Analogously, for every $k\in \mathbb{N}$, we have a measurable partition $\xi^u_k$ $\nu_k$-subordinated to $W^u_k$. 

Let $\pi\colon \mathbb{P}(E^c)\to M$, $F=Df\vert E^c$ and $\mathcal{F}^u$ by the foliation of $\mathbb{P}(F)$ given by Proposition \ref{Fu}. Define $\widehat{\mathcal{A}}=\pi^{-1}(\mathcal{A})$ and $\widehat{\mathcal{A}}^{u}$ as the partition defined by the intersection of elements of $\widehat{\mathcal{A}}$ and leaves of $\mathcal{F}^u$ as before. 

Define $$\widehat{\xi}^u=\bigvee_{j=0}^{\infty} \mathbb{P}(F^j)(\widehat{\mathcal{A}}^{u}),$$ then $\widehat{\xi}^u$ is a measurable partition of $\mathbb{P}(E^c)$ $m$-subordinated to $\mathcal{F}^u$ for every $\mathbb{P}(F)$-invariant probability measure $m$ projecting down to $\nu$. 

We can repeat this construction for every $k\in \mathbb{N}$ to obtain a measurable partition $\widehat{\xi}^u_k$ which is $m_k$-subordinated to $\mathcal{F}^u_ k$ for every $\mathbb{P}(F_k)$-invariant probability measure $m_k$ projecting down to $\nu_k$. Here $F_k=Df_k\vert E^c(f_k)$.

Observe that $\pi(\widehat{\xi}_u(x,v))=\xi_u(x)$ and $\pi_k(\widehat{\xi}_k^u(x,v))=\xi_k^u(x)$.

\begin{prop}\label{infn} If $f\in \mathcal{B}^N(M)$, $\nu$ is an $f$-invariant probability measure, $F=Df\vert E^c$, $\mathcal{F}^u$ is defined by Proposition \ref{Fu} and $m$ is a $\mathbb{P}(F)$-invariant probability measure projecting down to $\nu$, then 
\begin{equation*}
\frac{1}{n} H_m\left(\bigvee_{j=1}^n \mathbb{P}(F^{-j})(\widehat{\mathcal{A}}^u)\Big| \widehat{\mathcal{A}}^u\right)\searrow_{n\to \infty} h_m(\mathbb{P}(F),\mathcal{F}^u)
\end{equation*}
\end{prop} 
\begin{proof}
Using the invariance of $m$ and the following property of the entropy: $$H_{m}(\mathcal{P}\vee \mathcal{Q})=H_m(\mathcal{P}\vert \mathcal{R}) + H_m(\mathcal{Q}\vert \mathcal{P}\vee \mathcal{R}),$$ we have that 
\begin{equation}\label{prop54}
H_m\left(\bigvee_{j=1}^n \mathbb{P}(F^{-j})(\widehat{\mathcal{A}}^u)\Big| \widehat{\mathcal{A}}^u\right)=\sum_{j=1}^n H_m\left(\widehat{\mathcal{A}}^u \Big| \bigvee_{i=1}^j \mathbb{P}(F^{i})(\widehat{\mathcal{A}}^u) \right).
\end{equation} We refer the reader to the proof of Proposition 4.1 of \cite{Y} for more details about how to prove this identity. 

Since $$\bigvee_{i=1}^j \mathbb{P}(F^{i})(\widehat{\mathcal{A}}^u)\nearrow_{j\to \infty} \bigvee_{i=1}^{\infty} \mathbb{P}(F^{i})(\widehat{\mathcal{A}}^u),$$ and $H_m\left(\widehat{\mathcal{A}}^u \Big| \mathbb{P}(F)(\widehat{\mathcal{A}}^u) \right)<\infty$ (See Lemma 3.3 of \cite{Y}), we can apply Lemma \ref{limpart} to obtain, $$H_m\left(\widehat{\mathcal{A}}^u \Big| \bigvee_{i=1}^j \mathbb{P}(F^{i})(\widehat{\mathcal{A}}^u) \right)\searrow_{j\to \infty} H_m\left(\widehat{\mathcal{A}}^u \Big| \bigvee_{i=1}^{\infty} \mathbb{P}(F^{i})(\widehat{\mathcal{A}}^u) \right)=h_m(\mathbb{P}(F),\mathcal{F}^u).$$

Finally, Equation (\ref{prop54}) finishes the proof. 
\end{proof}

\subsection{Construction of a finite partition.}

The key argument to prove Proposition \ref{limitsup} is to approximate the entropies $h_{m}(\mathbb{P}(F),\mathcal{F}^{u})$ and $h_{m_k}(\mathbb{P}(F_k),\mathcal{F}^{u}_k)$ by the conditional entropy of two finite partitions. 

\subsubsection{An auxiliary partition}

Let $m$ be a $\mathbb{P}(F)$-invariant probability measure projecting down to $\mu$ and $\pi\colon \mathbb{P}(E^c)\to M$. 

For every $i\in \{1,...,q\}$, $D_i$ denotes the cross-section associated to the foliation boxes of $W^u$ as in Section 5.1. Moreover, $$\widehat{B}_i=\pi^{-1}(B_i) \text{ and } \widehat{D_i}=\pi^{-1}(D_i).$$ 

We are going to consider a sequence of finite partitions of $\widehat{D}_i$, denoted by $\widehat{\mathcal{C}}_{i,l}$, such that:
\begin{enumerate}[label=(\roman*)]
\item $\widehat{\mathcal{C}}_{i,l}<\widehat{\mathcal{C}}_{i,l+1}$ for every $l\in \mathbb{N}$. 
\item $\text{diam}(\widehat{\mathcal{C}}_{i,l})\to 0$ as $l\to \infty$. 
\item For any $i,l\geq 1$ and any element $\widehat{C}$ of $\widehat{\mathcal{C}}_{i,l}$, we have $$m\left(\bigcup_{(x,v)\in \partial \widehat{C}} \widehat{\mathcal{A}}^u (x,v)\right)=0.$$
\item For any $i,l\geq 1$, $\pi(\widehat{\mathcal{C}}_{i,l})$ is a partition of $D_i$. 
\end{enumerate}

One way to construct the partitions $\widehat{\mathcal{C}}_{i,l}$ is to consider a sequence of finite partitions on $D_i$, $\mathcal{C}_{i,1}< \mathcal{C}_{i,2}< \cdots$, such that $\text{diam}(\mathcal{C}_{i,l})\to 0$ as $l\to \infty.$ Then, we can use the local charts of the fiber bundle $\mathbb{P}(E^c)$ and define a finite partition on the fibers for $\pi^{-1}(\mathcal{C}_{i,l})$. Since in this case, the fibers are homeomorphic to $S^1$, we can take the finite partitions as intervals contained on $[0,1]$ with arbitrarily small diameter. 

\subsubsection{A finite partition of $\mathbb{P}(E^c)$.}

For every $i\in \{1,..,q\}$ and every $l\in \mathbb{N}$, the partition $\widehat{\mathcal{C}}_{i,l}$ defined above induces a partition on $\widehat{B}_i$: $$\widehat{\mathcal{C}}^{0}_{i,l}=\left\{\bigcup_{(x,v)\in \widehat{C}} \widehat{\mathcal{A}}^{u} (x,v): \widehat{C} \, \text{is an element of}\; \widehat{\mathcal{C}}_{i,l}\right\}.$$

Observe that by property (ii) in the definition of $\widehat{\mathcal{C}}_{i,l}$, if $(x,v)\in \widehat{B}_i$, we have that $\widehat{C}^0_{i,l}(x,v)\to \widehat{\mathcal{A}}^u (x,v)$ as $l\to \infty$.

\subsubsection{Final construction.}

For every $n\in \mathbb{Z}$, $n\geq 0$, let $\widehat{\mathcal{A}}^n=\bigvee_{j=0}^n \mathbb{P}(F^{-j})(\widehat{\mathcal{A}})$.

By the definition of $r_0$ and the property in the diameter of $\mathcal{A}$, we know that for every element $P\in \widehat{\mathcal{A}}^n$, there exists $i\in \{1,...,q\}$ such that $P\subset \widehat{B}_i$. Then, $\widehat{\mathcal{C}}^0_{i,l}$ induces a partition on $P$ whose elements we denote by $P_l$. 

The observation above implies that for every $l\in \mathbb{N}$, it is possible to define a partition of $\mathbb{P}(E^c)$ by, $$\widehat{\mathcal{A}}^n_l=\{P_l: P\in \widehat{\mathcal{A}}^n\}.$$

\begin{obs} We can repeat the same construction of Sections 5.3.1-5.3.3 for every $k\in \mathbb{N}$. For every $i\in \{1,...,q\}$ and $l\in \mathbb{N}$, we construct $\widehat{\mathcal{C}}^k_{i,l}$ which are partitions of $\widehat{B}^k_i=\pi_k^{-1}(B^k_i)$. Then, we use them to define partitions of $\mathbb{P}(E^c(f_k))$ as above. We denote these partitions by $\widehat{\mathcal{A}}^n_{k,l}$. 
\end{obs}
 
We have the following properties whose verification is direct form the definitions. Observe that 
\begin{equation}\label{idemu}
\widehat{\mathcal{A}}^n\vee \widehat{\mathcal{A}}^u=\bigvee_{j=0}^n \mathbb{P}(F^{-j})(\widehat{\mathcal{A}}^u).
\end{equation}

\begin{lemma}\label{propr} If $f\in \mathcal{B}^N(M)$, $\nu$ is an $f$-invariant probability measure, $F=Df\vert E^c$, $\mathcal{F}^u$ is defined by Proposition \ref{Fu} and $m$ is a $\mathbb{P}(F)$-invariant probability measure projecting down to $\nu$, then for every $n\in \mathbb{Z}$, $n\geq 0$,
\begin{enumerate}[label=\emph{(\alph*)}]
\item $\widehat{\mathcal{A}}^n_l\nearrow_{l\to \infty} \bigvee_{j=0}^n \mathbb{P}(F^{-j})(\widehat{\mathcal{A}}^u)$,
\vspace{0.2cm}
\item $\widehat{\mathcal{A}}^n<\widehat{\mathcal{A}}^n_l<\bigvee_{j=0}^n \mathbb{P}(F^{-j})(\widehat{\mathcal{A}}^u)$,
\vspace{0.2cm}
\item $m(\partial \widehat{\mathcal{A}}^n_l)=0$. 
\end{enumerate}
\end{lemma}

\subsection{Relation between the partitions defined above.}

The next proposition provides a relation between the entropy of the partition $\widehat{\mathcal{A}}^n_l$ and the entropy of the partition $\widehat{\mathcal{A}}^u$.

In order to simplify the notation, we are going to denote $\widehat{\mathcal{A}}^0_l=\widehat{\mathcal{A}}_l$ and $\widehat{\mathcal{A}}^0_{k,l}=\widehat{\mathcal{A}}_{k,l}$.

\begin{prop}\label{relac} If $f\in \mathcal{B}^N(M)$, $\nu$ is an $f$-invariant probability measure, $F=Df\vert E^c$, $\mathcal{F}^u$ is defined by Proposition \ref{Fu} and $m$ is a $\mathbb{P}(F)$-invariant probability measure projecting down to $\nu$, then for every $n\in \mathbb{N}$, 
\begin{equation*}
H_m(\widehat{\mathcal{A}}^n_l\vert \widehat{\mathcal{A}}_l)\searrow_{l\to \infty} H_m\left(\bigvee_{j=0}^n \mathbb{P}(F^{-j})(\widehat{\mathcal{A}}^u)\Big| \widehat{\mathcal{A}}^u\right).
\end{equation*}

In particular, by Proposition \ref{infn}, for every $n,l\in \mathbb{N}$, we have 
\begin{equation}\label{tres}
H_{m}(\widehat{\mathcal{A}}^n_{l}\vert \widehat{\mathcal{A}}_{l})\geq n\; h_{m}(\mathbb{P}(F), \mathcal{F}^u).
\end{equation}
\end{prop} 
\begin{proof}
We start with the following claim which is a key part of the proof. 
\begin{claim}
For every $l_1<l_2$, we have $$\widehat{\mathcal{A}}^n_{l_2}\vee \widehat{\mathcal{A}}_{l_2}=\widehat{\mathcal{A}}^n_{l_1}\vee \widehat{\mathcal{A}}_{l_2}.$$
\end{claim}
\begin{proof}
In the notation of subsection 5.3.3, since $\widehat{C}^{0}_{i,{l_1}}<\widehat{C}^{0}_{i,{l_2}}$, we have that for every $(x,v)\in \mathbb{P}(E^c)$, $$\widehat{C}^{0}_{i,{l_1}}(x,v)\cap\widehat{C}^{0}_{i,{l_2}}(x,v)=\widehat{C}^{0}_{i,{l_2}}(x,v).$$ 

Then, for every $(x,v)\in \mathbb{P}(E^c)$,
\begin{equation*}
\begin{aligned}
\widehat{\mathcal{A}}^n_{l_2}\vee \widehat{\mathcal{A}}_{l_2}(x,v)&=\widehat{\mathcal{A}}^n_{l_2}(x,v)\cap \widehat{\mathcal{A}}_{l_2}(x,v)\\
&=\widehat{\mathcal{A}}^n(x,v)\cap \widehat{C}^{0}_{i,{l_2}}(x,v)\cap\widehat{\mathcal{A}}(x,v)\cap \widehat{C}^{0}_{i,{l_2}}(x,v)\\
&=\widehat{\mathcal{A}}^n(x,v)\cap \widehat{C}^{0}_{i,{l_1}}(x,v)\cap\widehat{\mathcal{A}}(x,v)\cap \widehat{C}^{0}_{i,{l_2}}(x,v)\\
&=\widehat{\mathcal{A}}^n_{l_1}(x,v)\cap \widehat{\mathcal{A}}_{l_2}(x,v)\\
&=\widehat{\mathcal{A}}^n_{l_1}\vee \widehat{\mathcal{A}}_{l_2}(x,v).
\end{aligned}
\end{equation*}
\end{proof}

If we apply the claim above to $l_1=1$ and $l_2=l>1$, then we have that,
\begin{equation}\label{cuatro}
\begin{aligned}
H_m(\widehat{\mathcal{A}}^n_{l}\vert \widehat{\mathcal{A}}_{l})&=H_m(\widehat{\mathcal{A}}^n_{l}\vee \widehat{\mathcal{A}}_{l}\vert \widehat{\mathcal{A}}_{l})\\
                                                               &=H_m(\widehat{\mathcal{A}}^n_{1}\vee \widehat{\mathcal{A}}_{l}\vert \widehat{\mathcal{A}}_{l})\\
																															&=H_m(\widehat{\mathcal{A}}^n_{1}\vert \widehat{\mathcal{A}}_{l}).
\end{aligned}
\end{equation}

By the item (a) of Lemma \ref{propr} for $n=0$, we know that $\widehat{\mathcal{A}}_{l}\nearrow \widehat{\mathcal{A}}^u$. Moreover, since the partitions $\widehat{\mathcal{A}}^n_{l}$ are finite, we can use Lemma \ref{limpart}. By Equation (\ref{cuatro}), we obtain, 
\begin{equation*}
H_m(\widehat{\mathcal{A}}^n_{l}\vert \widehat{\mathcal{A}}_{l})\searrow_{l\to \infty} H_m(\widehat{\mathcal{A}}^n_{1}\vert \widehat{\mathcal{A}}^u)=H_m(\widehat{\mathcal{A}}^n_{1}\vee \widehat{\mathcal{A}}^u\vert \widehat{\mathcal{A}}^u)
\end{equation*}

In order to conclude the proposition, it is enough to prove the following claim.

\begin{claim}
$$\widehat{\mathcal{A}}^n_{1}\vee \widehat{\mathcal{A}}^u=\bigvee_{j=0}^n \mathbb{P}(F^{-j})(\widehat{\mathcal{A}}^u).$$ 
\end{claim}
\begin{proof}\let\qed\relax
By property (b) of Lemma \ref{propr} for $l=1$, we have 
\begin{equation*}
\widehat{\mathcal{A}}^n<\widehat{\mathcal{A}}^n_{1}<\bigvee_{j=0}^n \mathbb{P}(F^{-j})(\widehat{\mathcal{A}}^u).
\end{equation*}

Then, 
\begin{equation*}
\widehat{\mathcal{A}}^n\vee \widehat{\mathcal{A}}^u<\widehat{\mathcal{A}}^n_{1}\vee \widehat{\mathcal{A}}^u<\bigvee_{j=0}^n \mathbb{P}(F^{-j})(\widehat{\mathcal{A}}^u).
\end{equation*}

Using Equation (\ref{idemu}), we  conclude the claim and therefore the proposition. 
\end{proof}
\end{proof}
																															
\subsection{Proof of Proposition \ref{limitsup}}
First, we recall the statement of the proposition.

\begin{propf}
Fix $N>0$. Let $f_k,f\in \mathcal{B}^N(M)$ for every $k\in \mathbb{N}$, $f_k\to f$ in the $C^1$ topology, $F=Df\vert E^c$, $F_k=Df_k\vert E^c(f_k)$. 

Suppose $\nu_k$ is an $f_k$-invariant probability measure for every $k\in \mathbb{N}$ and $\nu$ is an $f$-invariant probability measure.

If $m_k$ is a $\mathbb{P}(F_k)$-invariant probability measure projecting down to $\nu_k$ and $m_k\to m$ in the weak$^{*}$ topology, then
\begin{equation*}
\limsup_{k\to \infty} h_{m_k}(\mathbb{P}(F_k), \mathcal{F}^{u}_k)\leq h_{m}(\mathbb{P}(F), \mathcal{F}^{u}).
\end{equation*}
\end{propf} 
\begin{proof}
The key observation is that for every $n$ and $l$ fixed, 
\begin{equation}\label{limfinito}
\lim_{k\to \infty} H_{m_k}(\widehat{\mathcal{A}}^{n}_{k,l}\vert \widehat{\mathcal{A}}_{k,l})=H_m(\widehat{\mathcal{A}}^{n}_{l}\vert \widehat{\mathcal{A}}_{l}).
\end{equation}

This is a consequence of the fact that all the partitions we are considering are finite and the following property: for every $P\in \widehat{\mathcal{A}}^{n}_{l}$ there exists a sequence $P_k\in \widehat{\mathcal{A}}^{n}_{k,l}$ such that $$\lim_{k\to \infty} m_k(P_k)=m(P).$$ In order to see this, observe that $m_k\to m$ in the weak$^*$ topology and since $f_k\to f$ in the $C^1$ topology, we have Proposition \ref{contF}. That is, for every $P\in \widehat{\mathcal{A}}^{n}_{l}$, we have subsets $P_k\in \widehat{\mathcal{A}}^{n}_{k,l}$ such that $P_k\to P$ in the Hausdorff topology. Here we are using property (c) in Lemma \ref{propr} and the fact that $\widehat{\mathcal{A}}$ and the sections $D_i$ are the same for every $k\in \mathbb{N}$.  

By Proposition \ref{infn}, for every $\epsilon>0$ there exists $n_0\in \mathbb{N}$ such that 
\begin{equation}\label{uno} 
\begin{aligned}
\frac{1}{n_0} H_m\left(\bigvee_{j=0}^{n_0} \mathbb{P}(F^{-j})(\widehat{\mathcal{A}}^u)\Big| \widehat{\mathcal{A}}^u\right)-\frac{\epsilon}{3}&=\\
\frac{1}{n_0} H_m\left(\bigvee_{j=1}^{n_0} \mathbb{P}(F^{-j})(\widehat{\mathcal{A}}^u)\Big| \widehat{\mathcal{A}}^u\right)-\frac{\epsilon}{3}& \leq h_{m}(\mathbb{P}(F),\mathcal{F}^u).
\end{aligned}
\end{equation}

Moreover, by Proposition \ref{relac} for $n=n_0$, there exists $l_0\in \mathbb{N}$ such that 
\begin{equation}\label{dos}
H_m(\widehat{\mathcal{A}}^{n_0}_{l_0}\vert \widehat{\mathcal{A}}_{l_0})- \frac{\epsilon}{3}\leq H_m\left(\bigvee_{j=0}^{n_0} \mathbb{P}(F^{-j})(\widehat{\mathcal{A}}^u)\Big| \widehat{\mathcal{A}}^u\right).
\end{equation}

By Equation (\ref{limfinito}) applied to $n_0$ and $l_0$, there exists $k_0\in \mathbb{N}$ such that for every $k\geq k_0$, 
\begin{equation}\label{siete}
H_{m_k}(\widehat{\mathcal{A}}^{n_0}_{k,l_0}\vert \widehat{\mathcal{A}}_{k,l_0})- \frac{\epsilon}{3}\leq H_m(\widehat{\mathcal{A}}^{n_0}_{l_0}\vert \widehat{\mathcal{A}}_{l_0}).
\end{equation}

Then, by Equation (\ref{uno}), Equation (\ref{dos}) and Equation (\ref{siete}), we have that for every $k\geq k_0$, 
\begin{equation}\label{cinco}
\frac{1}{n_0} H_{m_k}(\widehat{\mathcal{A}}^{n_0}_{k,l_0}\vert \widehat{\mathcal{A}}_{k,l_0})-\epsilon\leq h_{m}(\mathbb{P}(F),\mathcal{F}^u).
\end{equation}

Finally, Equation (\ref{tres}) applied to $f_k$ and Equation (\ref{cinco}) give, 
$$h_{m_k}(\mathbb{P}(F_k), \mathcal{F}^u_k)-\epsilon\leq h_{m}(\mathbb{P}(F), \mathcal{F}^u).$$
Then, for every $\epsilon>0$ there exists $k_0\in \mathbb{N}$ such that for every $k\geq k_0$, the above inequality is verified. This concludes the proposition.  
\end{proof}

\section{Characterization of discontinuity points} 

The results in this section are classical in the setting of linear cocycles over partially hyperbolic maps. We are able to adapt their proofs using Proposition \ref{limitu}. 


Let $N>0$, $f\in \mathcal{B}^N(M)$ and $F=Df\vert E^c$. From now on we fix the Riemannian metric given by Equation (\ref{ph}).


If $\eta\colon M\to \mathbb{R}$ is defined by $\eta(x)=|\det\, F_x|^{-1/2}$, then we can consider a new cocycle over $f$ by $F'_x=\eta(x)\cdot F_x$. Notice that $\left|\det\, F'_x\right|=1$ for every $x\in M$ and the extremal Lyapunov exponents of $F'$ for an ergodic measure of $f$, $\nu$, satisfy the following,
\begin{equation}\label{lyap0}
\begin{aligned}
\lambda_{\pm}(F', \nu)&=\lambda_{\pm}(F, \nu) + \int{ \log \left|\eta(x)\right| \, d\nu}\\
                   &=\lambda^c_{1,2}(f, \nu) + \int{ \log \left|\eta(x)\right| \, d\nu}.
\end{aligned}
\end{equation}

\begin{equation}\label{lyap1}
\lambda_{+}(F', \nu)+\lambda_{-}(F', \nu)=0.
\end{equation}

\begin{obs}\label{idemproj}
Notice that $\mathbb{P}(F)=\mathbb{P}(F')$. 
\end{obs}

\begin{prop}[Proposition 4.6, \cite{LMY}] \label{nozero} Let $N>0$. If $f\in \mathcal{B}^N(M)$, $F'=\eta\cdot F$, $\nu$ is an ergodic probability measure of $f$ and $\lambda_{+}(F', \nu)> 0 > \lambda_{-}(F', \nu)$, then there exist two $\mathbb{P}(F)$-invariant probability measures projecting down to $\nu$ denoted by $m^{+}$ and $m^{-}$, which are a u-state and an s-state for $\nu$ respectively. Moreover, if $m$ is any $\mathbb{P}(F)$-invariant probability measure projecting down to $\nu$, then there exists $t\in [0,1]$ such that $$m=t\, m^{+} + (1-t)\, m^{-}.$$
\end{prop} 


Given $f\in \mathcal{B}^N(M)$ and $\nu$ an ergodic probability measure for $f$, we say that $f$ is a \textit{discontinuity point for the center Lyapunov exponents} if there exist a sequence $f_k\in \mathcal{B}^N(M)$ and probability measures $\nu_k$ which are ergodic for $f_k$ such that $f_k\to f$ in the $C^1$ topology, $\nu_k\to \nu$ in the weak$^{*}$ topology and for $i=s$ or $i=u$ verifies

\begin{equation*}
\lim_{k\to \infty} h_{\nu_k}(f_k, W^{i}_k)=h_{\nu}(f, W^{i})
\end{equation*}

\begin{equation*}
\text{and} \quad \left(\lambda^c_1(f_k,\nu_k),\lambda_2^c(f_k, \nu_k)\right)\; \text{ does not converge to }\; \left(\lambda^c_1(f,\nu),\lambda_2^c(f, \nu)\right).
\end{equation*}

\begin{prop}\label{graldiscont} Fix $N>0$. Let $f\in \mathcal{B}^N(M)$, $\nu$ an ergodic probability measure for $f$ and $F=Df\vert E^c$. If $f$ is a discontinuity point for the center Lyapunov exponents, then there exists a $\mathbb{P}(F)$-invariant probability measure $m$ projecting down to $\nu$ which is an su-state.
\end{prop}
\begin{proof}
In the following, we give the proof for the case that the partial entropies along the unstable foliations converge. The other case is analogous considering $f^{-1}$.  

Since the functions $(f, \nu) \mapsto \lambda^c_1(f, \nu)$ and $(f, \nu)\mapsto \lambda^c_2(f, \nu)$ are upper semi-continuous and lower semi-continuous, respectively, the discontinuity of $\lambda^c_{i}(f, \nu)$ for some $i\in \{1,2\}$ implies that $\lambda^c_1(f, \nu)\neq \lambda^c_2(f, \nu)$. Therefore, by Equation (\ref{lyap0}), $\lambda_{+}(F', \nu)\neq \lambda_{-}(F', \nu)$ and by Equation (\ref{lyap1}), we have $\lambda_{+}(F', \nu)> 0 > \lambda_{-}(F', \nu)$. 

Let $m^{+}$ and $m^{-}$ be given by Proposition \ref{nozero}. If $\Phi$ is given by Lemma \ref{fust}, then $\lambda_{+}(F', \nu)=\int{\Phi(x,v)\, d\,m^{+}}.$

Consider now the cocycle $F'_k=\eta_k \cdot Df_k\vert E^c(f_k)$ associated to $f_k$ and let $m_k$ be an ergodic probability measure for $\mathbb{P}(F_k)$ projecting down to $\nu_k$ which realizes the maximum in Lemma \ref{fust}, that is, $\lambda_{+}(F'_k, \nu_k)=\int{ \Phi_k(x,v)\, d\, m_k}$.

By the hypotheses and Equation (\ref{lyap0}), $\lambda_{+}(F'_k, \nu_k)$ does not converge to $\lambda_{+}(F', \nu)$.

The measure $m_k$ is a $u$-state for every $k\in \mathbb{N}$, this is a consequence of Proposition \ref{nozero} if $\lambda_{+}(F'_k, \nu_k)>0$ and a consequence of Theorem 4.1 of \cite{ASV} if $\lambda_{+}(F'_k, \nu_k)=0$. 

There exist a subsequence $k_j$ and a measure $m$ in $\mathit{\mathbb{P}(TM)}$ such that $m_{k_j}\rightarrow m$ in the weak$^*$ topology. Therefore, by Proposition \ref{limitu}, the limit measure $m$ is a $u$-state for $\mathbb{P}(F)$ and $\nu$. 

Moreover, since $f_k\to f$ we have $\int{ \Phi_{k_j}(x,v)\, d\, m_{k_j}}\to \int{ \Phi(x,v)\, d\, m}$. On the other hand, since $\lambda_{+}(F'_k, \nu_k)$ does not converge to $\lambda_{+}(F', \nu)$, $$\lim \limits_{k_j} \int{ \Phi_{k_j}(x,v)\, d\, m_{k_j}} < \lambda_{+}(F'\, \nu)=\int{\Phi(x,v)\, d\,m^{+}}.$$ 

These properties allow us to conclude that $m$ is a $\mathbb{P}(F)$-invariant probability measure projecting down to $\nu$ which is a $u$-state and it is different from $m^{+}$. Therefore, by Proposition \ref{nozero}, there exists $t\neq 1$ such that $m= t\, m^{+} + (1-t)\, m^{-}$. Now, we can write $m^{-}=\frac{m- t\, m^{+}}{(1-t)}$. Moreover, we know that $m^{+}$ is a u-state and $m^{-}$ an s-state. This implies, $m^{-}$ is an $su$-state. 

\end{proof}

Applying this proposition to the case of volume-preserving diffeomorphism we are able to conclude the following corollary. 

Recall that if we are in the symplectic context, $\mu$ always denotes the Lebesgue measure associated to $\omega$ and for $*\in \{\mu, \omega\}$, $L_{*}\colon \mathcal{E}^N_{*}(M)\to \mathbb{R}^2$ is defined by $$L_{*}(f)=(\lambda^c_1(f,\mu), \lambda^c_2(f, \mu)).$$

\begin{cor}\label{discont} Fix $N>0$. Let $*\in \{\mu, \omega\}$, $f\in \mathcal{E}^N_{*}(M)$ and $F=Df\vert E^c$. If $f$ is a discontinuity point for $L_{*}$, then every $\mathbb{P}(F)$-invariant probability measure $m$ projecting down to $\mu$ is an su-state.
\end{cor}
\begin{proof}
By Equation (\ref{Parte2}), we know that if $f$ is a discontinuity point for $L_{*}$, then it is a discontinuity point for the center Lyapunov exponents in the sense above. Therefore, since both partial entropies converge, we can conclude that $m^{+}$ and $m^{-}$ are both $su$-states. This together with Proposition \ref{nozero} finish the proof. 
\end{proof}

\section{Proof of the theorems}

\subsection{Theorem A and A'}


For the proof of Theorem A and Theorem A' we use the arguments of Theorem 6.1 of \cite{LMY}. In the following, we give an outline in order to explain the main ideas. 

We denote $\lambda_i^c(f,\mu)=\lambda_i^c(f)$ for $i\in \{1,2\}$.  

\begin{theorem}\label{open} Let $N>0$ and $*\in \{\mu, \omega\}$. Suppose $f_k\to f$ in the $C^1$ topology and $f_k, f\in \mathcal{E}^N_{*}(M)$. If $\lambda^c_1(f_k)=\lambda^c_2(f_k)$ for every $k\in \mathbb{N}$, then $\lambda^c_1(f)=\lambda^c_2(f)$.
\end{theorem}

It is clear that this theorem will imply Theorem A and Theorem A'. 

Observe that in the symplectic case (Theorem A) the assumption of $\lambda^c_1(f)\neq\lambda^c_2(f)$ implies that $f$ is non-uniformly hyperbolic and therefore, by \cite{K}, it has a pinching hyperbolic periodic point. This is not true in the volume-preserving setting and this is the reason why we need to ask for it in the hypotheses. In the following outline it will be clear that the existence of a pinching hyperbolic periodic point is essential for the proof. 

\vspace{0.3cm}
\emph{Outline of the proof of Theorem \ref{open}.}
Let $N>0$ and $*\in \{\mu, \omega\}$. Let $f_k\to f$ in the $C^1$ topology, $f_k\in \mathcal{E}^N_{*}(M)$ and $\lambda^c_1(f_k)=\lambda^c_2(f_k)$ for every $k\in \mathbb{N}$. Assume that $f\in \mathcal{E}^N_{*}(M)$ and $\lambda^c_1(f)\neq \lambda^c_2(f)$. 

By the hypotheses, $f$ is a discontinuity point for $L_{*}$ (see Definition \ref{contin}). Moreover, if $F=Df\vert E^c$ and $F'=\eta\cdot F$, by Equations (\ref{lyap0}) and (\ref{lyap1}), $\lambda_{+}(F')>0>\lambda_{-}(F')$. See the argument in the second paragraph of the proof of Proposition \ref{graldiscont}. Consider $m^{+}$ and $m^{-}$ given by Proposition \ref{nozero} applied to $f$ and $F$.

On the other hand, if $F_k=Df_k\vert E^c(f_k)$ and $F'_k=\eta_k\cdot F_k$, then $\lambda_{\pm}(F'_k)=0$ for every $k\in \mathbb{N}$. This is again a consequence of Equations (\ref{lyap0}) and (\ref{lyap1}). For every $k\in \mathbb{N}$, fix $m_k$ any ergodic $\mathbb{P}(F_k)$-invariant probability measure projecting down to $\mu$.


By the version of the Invariance Principle stated in Theorem B of \cite{ASV}, we have a disintegration $\{m_{k,x}: x\in M\}$ for every $k\in \mathbb{N}$ which is invariant by stable and unstable holonomies and such that the function $x\mapsto m_{k,x}$ depends continuously on the base point $x\in M$. 

Since $f$ has a pinching hyperbolic periodic point, for $k$ big enough the same property holds for $f_k$. This implies that $m_{k,x}$ is atomic for every $x\in M$ and it has at most two atoms. 


First we suppose that for every $k$ big enough, $m_{k,x}$ has exactly two atoms. The result about the continuity of the holonomies given by Proposition \ref{continuity} and the fact that the periodic point varies continuously with $f_k$ allow us construct two sections of $\mathbb{P}(E^c(f_k))$, $x\mapsto a_{k,x}$ and $x\mapsto b_{k,x}$, with the following properties:

\begin{enumerate}[label=(\alph*)] 
\item For every $x\in M$, $m_{k,x}= t \delta_{a_{k,x}} + (1-t) \delta_{b_{k,x}}$.
\item For every $x\in M$, $\mathbb{P}(F_{k,x})(a_{k,x})=a_{k, f_k(x)}$ and $\mathbb{P}(F_{k,x})(b_{k,x})=b_{k, f_k(x)}$.
\item $x\mapsto a_{k,x}$ and $x\mapsto b_{k,x}$ vary continuously with the point $x\in M$. 
\item For every $\epsilon>0$, there exists $K\in \mathbb{N}$ such that $d(a_{k,x}, supp\, m^{+}_x)<\epsilon$ and $d(b_{k,x}, supp\, m^{-}_x)<\epsilon$ for every $k\geq K$ and every $x\in M$.
\end{enumerate}

Properties (a) to (c) allow us to define two $\mathbb{P}(F_k)$-invariant probability measures projecting down to $\mu$ by $$m_k^{+}=\int \delta_{a_{k,x}}\, d\mu \quad \text{and}\quad m_k^{-}=\int \delta_{b_{k,x}}\, d\mu.$$ 

Therefore, $m_k$ can be written as $m_k= t\, m_k^{+} + (1-t)\, m_k^{-}$. This is a contradiction, since we chose $m_k$ to be ergodic. 

This argument shows that there exists a subsequence $k_j$ such that $m_{k_j,x}$ can have at most one atom for every $k_j$. Proceeding as above, we obtain a property analogous to (d) between $m_{k_j}$ and $m^{+}$ or $m^{-}$. More precisely, in the first case we obtain that for every $\epsilon>0$, there exists $J\in \mathbb{N}$ such that $d(supp\, m_{k_j,x}, supp\, m^{+}_x)<\epsilon$ for every $k_j\geq J$ and every $x\in M$. The other case is analogous with $m^{-}$ instead of $m^{+}$. Therefore we can conclude that $m_{k_j}\to m^{+}$ or $m_{k_j}\to m^{-}$. 

Suppose $m_{k_j}\to m^{+}$, the other case is analogous. Since $f_k\to f$, we have $0=\lambda_{+}(F'_{k_j})=\int{ \Phi_{k_j}(x,v)\, d\, m_{k}}\to \int{ \Phi(x,v)\, d\, m^{+}}=\lambda_{+}(F')$. However, we were assuming that $\lambda^c_1(f)\neq \lambda^c_2(f)$ which implies, by Equation (\ref{lyap0}) and Equation (\ref{lyap1}), that $\lambda_{+}(F')>0$. Therefore, the conclusion we obtain is a contradiction. Finally, we conclude $\lambda^c_1(f)$ must be equal to $\lambda^c_2(f)$ as we wanted to prove. 
\qed

\subsection{Theorem B and B'}

We are going to prove the following theorem:

\begin{theorem}\label{interior} Let $N>0$ , $*\in \{\mu, \omega\}$ and $f\in \mathcal{E}^N_{*}(M)$. If $f$ is a discontinuity point for $L_{*}$, then $f$ can be $C^r$-approximated by diffeomorphisms which are bundle-free for $\mu$.
\end{theorem}

Recall that $L_{*}\colon \mathcal{E}^N_{*}(M)\to \mathbb{R}^2$ is defined by $$L_{*}(f)=(\lambda^c_1(f,\mu), \lambda^c_2(f, \mu)).$$ By Definition \ref{bundlefree}, a diffeomorphism $g$ is bundle-free for $\mu$ if the projective cocycle $\mathbb{P}(G)$ does not admit invariant probability measure projecting down to $\mu$ which are $su$-states. Here $G=Dg\vert E^c(g)$. 

Observe that by Corollary \ref{discont}, if $g$ is bundle-free, then it is a continuity point of $L_{*}$. Moreover, by Corollary \ref{limitu}, the property of not admitting an $su$-state is a $C^1$-open condition in $\mathcal{E}^N_{*}(M)$. This is a consequence of the fact that having a $u$-state is a closed property and the analogous property for $s$-states obtained applying Corollary \ref{limitu} to $f^{-1}$. Therefore, Theorem \ref{interior} implies Theorem B and Theorem B'. 

In the proof of Theorem A, we observed that if $f\in \mathcal{E}^N_{\omega}(M)$ is a discontinuity point for $L_{*}$, then $f$ has a pinching hyperbolic periodic point. We are using this property in the following proof.

\subsection*{Proof of Theorem \ref{interior}.}

In order to prove this theorem, we use the arguments of \cite{M}.

Let $N>0$ and $*\in \{\mu, \omega\}$. Suppose $f\in \mathcal{E}^N_{*}(M)$ and $f$ is a discontinuity point for $L_{*}$. Then, by Corollary \ref{discont}, we have that every $\mathbb{P}(F)$-invariant measure $m$ projecting down to $\mu$ is an $su$-state. Then, by Theorem E of \cite{ASV}, for every such $m$, we have a disintegration $\{m_{x}: x\in M\}$ which is invariant by stable and unstable holonomies and such that the function $x\mapsto m_{x}$ depends continuously on the base point $x\in M$. 

We fix $\epsilon>0$ small enough such that if $dist_{C^r}(g,f)<\epsilon$, then $g\in \mathcal{E}^N_{*}(M)$. We construct a sequence of perturbations $f_k$ for $f$ and $\epsilon$, using Lemma 4.1 of \cite{M}. As a consequence we have that $f_k\in \mathcal{E}^N_{*}(M)$ and $f_k\to f$ in the $C^1$ topology. Moreover, if $p$ a the pinching hyperbolic periodic point of $f$ previously fixed, we ask for the support of the perturbation to be disjoint of the orbit of $p$. That is, $f^j(x)=f^j_k(x)$ for every $x$ in a neighborhood of $p$, every $k\in \mathbb{N}$ and any $j\in \mathbb{Z}$. 

By contradiction, we suppose for every $k\in \mathbb{N}$, the derivative cocycle $\mathbb{P}(F_k)$, $F_k=Df_k\vert E^c(f_k)$, admits some $su$-state $m_k$. Again, by Theorem E of \cite{ASV}, we have a disintegration $\{m_{k,x}: x\in M\}$ for every $k\in \mathbb{N}$ which is invariant by stable and unstable holonomies and such that the function $x\mapsto m_{k,x}$ depends continuously on the base point $x\in M$. 

The continuity of $m_{k,x}$ and the invariance of $m_k$ implies $\mathbb{P}(F_{k,x})_{*}m_{k,x}=m_{k,f_k(x)}$ for every $x\in M$ and every $k\in \mathbb{N}$. Therefore, if $a,b\in \mathbb{P}(E^c_p)$ denotes the two elements associated to the eigenvectors of $Df^{n_p}_p\vert E^c_p$, we have the following three possibilities:

\begin{enumerate}[label=(\alph*)] 
\item There exists $k_j\to \infty$ such that $supp\, m_{k_j,p}=\{a\}$.
\item There exists $k_j\to \infty$ such that $supp\, m_{k_j,p}=\{b\}$.
\item There exists $k_j\to \infty$ such that $supp\, m_{k_j,p}=\{a,b\}$.
\end{enumerate}

If we are in cases (a) and (b), we can repeat the argument in the proof of Theorem A and A' to get $m_{k_j}\to m^{+}$ in the first case and $m_{k_j}\to m^{-}$ in the second case. Then, if we choose $m$ such that $m_{k_j}\to m$, in any of the three cases above we have $supp\, m_p\subset supp\, m_{k_j,p}$. 

Since, we have $supp\, m_p\subset supp\, m_{k_j,p}$, we can repeat the argument in the proof of Theorem B of \cite{M}. The main point is that the variation in the holonomies as a function of $f_k$ is exponentially small on $k$, although the size of the perturbations is polynomial in $k$. This allow to brake the rigidity given by Theorem E of \cite{ASV} and therefore we obtain a contradiction which comes from the assumption that there exist measures $m_k$ which are $su$-states for every $k\in \mathbb{N}$. 

\qed

\subsection{Theorem C and D}

It is clear from the statements that Theorem D implies Theorem C. Moreover, Theorem D is a consequence of Proposition \ref{graldiscont}: 

If $(f_k,\nu_k)$ in the statement of Theorem D does not verify $$\left(\lambda^c_1(f_k,\nu_k),\lambda_2^c(f_k, \nu_k)\right)\rightarrow \left(\lambda^c_1(f,\nu),\lambda_2^c(f, \nu)\right),$$ then Proposition \ref{graldiscont} implies the existence of a $\mathbb{P}(F)$-invariant measure projecting down to $\nu$ which is an $su$-state. This is a contradiction to the hypothesis of $f$ being bundle-free for $\nu$. 

Therefore, $$\left(\lambda^c_1(f_k,\nu_k),\lambda_2^c(f_k, \nu_k)\right)\rightarrow \left(\lambda^c_1(f,\nu),\lambda_2^c(f, \nu)\right),$$ and since $\nu$ is hyperbolic, for $k$ big enough the measures $\nu_k$ are also hyperbolic.

\end{document}